\documentclass{amsart}
\allowdisplaybreaks
\usepackage{amsfonts}
\usepackage{}
\usepackage{amsmath}
\usepackage{paralist}
\usepackage{amssymb}
\usepackage{amsthm}
\usepackage{amscd}
\usepackage{graphicx,mathrsfs}
\usepackage[colorlinks=true]{hyperref}
\hypersetup{urlcolor=blue, citecolor=red}

\numberwithin{equation}{section}
\newtheorem{Thm}{Theorem}[section]
\newtheorem{Lem}{Lemma}[section]
\newtheorem{Prop}{Proposition}[section]

\newtheorem{Def}{Definition}[section]
\newtheorem{Cor}{Corollary}[section]
\newtheorem{example}{Example}[section]
\newtheorem{Rem}{Remark}[section]

\begin{document}
\title{Mean Field Games with H\"{o}rmander diffusions}
\thanks{We greatly appreciate the anonymous referees for their valuable comments and suggestions. This work is funded by the National Natural Science Foundation of China (No. 12271269, No. 12326318 and No. 12471141), Shandong Provincial Natural Science Foundation (No. ZR2026QC1077), and Shandong Postdoctoral Science Foundation (No. SDZZ-ZR-202501395).
All authors contributed equally to this work.}
\author{Yiming Jiang}
\address{School of Mathematical Sciences and LPMC\\ Nankai University\\ Tianjin 300071 China}
\email{ymjiangnk@nankai.edu.cn}
\author{Jingchuang Ren}
\address{Yau Mathematical Sciences Center\\ Tsinghua University\\ Beijing 100084 China}
\email{ren\_math@tsinghua.edu.cn}
\author{Yawei Wei}
\address{School of Mathematical Sciences and LPMC\\ Nankai University\\ Tianjin 300071 China}
\email{weiyawei@nankai.edu.cn}
\author{Jie Xue}
\address{School of Mathematical Sciences\\ Ocean University of China\\ Qingdao 266100 China}
\email{jiexue@ouc.edu.cn}
\keywords{Mean field games; Degenerate parabolic equations; H\"{o}rmander vector fields; A priori Schauder estimates}
	
\subjclass[2020]{35K65; 35Q89; 49N80}
\begin{abstract}
In this paper, we study a class of degenerate mean field games (MFG, for short) systems with H\"{o}rmander diffusion, in which  the typical agent can move only along admissible direction. We establish the well-posedness of the MFG systems in intrinsic H\"{o}lder spaces, which describes the Nash equilibria for a differential game with infinitely many small players. The analysis builds upon degenerate parabolic equations defined on $\mathbb{T}^d$ induced by H\"{o}rmander vector fields without any group structure, for which we develop a global regularity theory for the first time in this general setting. A central difficulty arises from the anisotropic geometry of the state space, which induces non-commutativity and inhomogeneity of the underlying vector fields. Instead of the classical fundamental solution framework, we present an alternative method to derive a priori Schauder estimates for general H\"{o}rmander degenerate parabolic equations via Campanato spaces.
\end{abstract}
\maketitle

\section{INTRODUCTION}	
\subsection{Statement of the problem and motivation}
In this paper, our first aim is a priori Schauder estimates of the degenerate Cauchy problems as follows:
\begin{equation}\label{peq}
\left\{
  \begin{aligned}
     & \mathcal {H}u(t,x)=f(t,x),&(t,x)\in (0,T]\times\mathbb{T}^d,  \\
    &u(0,x)=g(x)& x\in\mathbb{T}^d,
  \end{aligned}
\right.
\end{equation}
where $\mathbb{T}^d=\mathbb{R}^d/\mathbb{Z}^d$, $d\geq3$, is
the standard $d$-dimensional torus, the parabolic non-divergence operator $\mathcal {H}$ is given by
\begin{equation}\label{001}
  \mathcal {H}:=\partial_t-\Delta_\mathcal {X}+c(t,x).
\end{equation}
 Here $\Delta_\mathcal {X}=\sum^{q}_{k=1}X_k^2$ is the Laplacian with respect to a set of vector fields $\mathcal {X}=\{X_1,X_2,\ldots,X_q\}$, $2\leq q<d$, defined on $\mathbb{T}^d$ with the form
\begin{equation}\label{sigma}
  X_k=\sum^d_{i=1}\sigma^{ik}(x)\partial_{x_i},\ \text{for}\ k=1,\ldots,q,
\end{equation}
where $\sigma^{ik}(\cdot)$  is smooth and possibly vanishing such that $\mathcal {X}$ satisfies the H\"{o}rmander's condition (see Definition \ref{HC}).  A fundamental property of H\"{o}rmander vector fields is that, at any point, the subspace spanned by their ellipticity directions has dimension strictly less than that of the state space, while the remaining directions are generated by commutators. Given any finite time $T>0$, the coefficients  $c(t,x)$, the functions $f(t,x)$  and $g(x)$  belong to intrinsic H\"{o}lder spaces for some $\alpha\in(0,1)$. Additional details can be found in Definition \ref{2.06}.

The second aim of this work is to apply the aforementioned regularity results to  study the well-posedness and regularity for the following MFG systems
\begin{equation}\label{MFG}
\left\{
  \begin{array}{ll}
    -\partial_t u-\Delta_\mathcal {X} u+H(x,Xu)=F(x,m),\qquad\quad & (t,x)\in [0,T)\times\mathbb{T}^d,\ \rm{(HJE)}\\
    \partial_t m-\Delta_\mathcal {X} ^*m-{\rm{div}}_{\mathcal {X}^*}(mD_pH(x,Xu))=0,\ \ & (t,x)\in (0,T]\times\mathbb{T}^d,\ \rm{(FPE)}\\
   u(T,x)=G(x,m_T),\  m(0,x)=m_0,\qquad\quad\ \ &x\in \mathbb{T}^d,\qquad\qquad\ \
  \end{array}
\right.
\end{equation}
which describes a differential game with larger number of small players.
  Here $X=(X_1,\ldots,X_q)$ is the corresponding subgradient vector associated to the vector fields $\mathcal {X}$ given by \eqref{sigma}, $\Delta_\mathcal {X}^*$ and ${\rm{div}}_{\mathcal {X}^*}$ are respectively the Laplacian and the divergence operator induced by the dual vector fields
\begin{equation}\label{ck}
X_k^*=-X_k+c_k,\ c_k=-\sum^{d}_{i=1}\partial_{x_i} \sigma^{ik}(x),\ \ k\in1,2,\ldots,q,
\end{equation}
$m_0\in\mathcal {P}(\mathbb{T}^d)$ is the initial distribution of the agents, and $\mathcal {P}(\mathbb{T}^d)$ denotes the space of Borel probability measures on $\mathbb{T}^d$ with finite first moments.

From the perspective of a single agent, the introduction of H\"{o}rmander vector fields $\mathcal {X}$ means that the agent can move only along admissible directions: a subspace of the tangent space. The MFG system \eqref{MFG} consists of a Hamilton-Jacobi equation (HJE, for short) for  $u(t,x)$ and a Fokker-Planck equation (FPE, for short) for $m(t,x)$. 
Note that the evolution of the value function $u(t,x)$ of a typical player is coupled to the mean-field density $m(t,x)$ of all agents, while the dynamics of the density $m(t,x)$ is in turn driven by the coefficients arising from $D_pH(x,Xu)$.


The heuristic interpretation of the MFG system \eqref{MFG} is the following. An average player controls the stochastic differential equation (SDE, for short)
\begin{equation}\label{SDE}
  \left\{
 \begin{array}{ll}
dZ_s=\sigma(Z_s)\alpha_s ds+\sqrt{2}\sigma(Z_s)dB_s,\\
    Z_t=x,
     \end{array}
\right.
\end{equation}
where $Z_s=(Z_s^1,\ldots,Z_s^d)^{T}$ is the $d$-dimensional state variable, $x\in\mathbb{T}^{d}$ is some fixed initial condition, $\sigma=\{\sigma^{ik}\}:[0,T]\times\mathbb{T}^d\to\mathbb{T}^{d\times q}$, is the intensity of the individual noise possibly vanishing, and $B_s=(B^1_s,\ldots,B^q_s)^{T}$ is a standard $q$-dimensional independent Brownian motion on a filtered probability space $(\mathbf{\Omega},\mathscr{F},\{\mathscr{F}_t\}_{t\geq0},\mathbf{P})$ satisfying the usual conditions in the stochastic analysis, see Chapter 3 in \cite{P09}. The control $\alpha=(\alpha_1,\ldots,\alpha_q)^{T}:[0,T]\times\mathbf{\Omega}\to\mathrm{A}$ are measurable $\{\mathscr{F}_t\}_{t\geq0}$-adapted maps taking values in some metric space $\mathrm{A}$, while $\mathcal {A}$ denotes the set of admissible controls. Assume that the coefficient $\sigma$ is Lipschitz continuous in the space variable $Z_s$, then the typical player aims at minimizing the cost function, the value function corresponds to
\begin{equation}\label{4.21}
u(t,x)=\inf_{\alpha\in\mathcal {A}}\mathbf{E}\Big[\int^T_tH^*(Z_s,\alpha_s)+F(Z_s,m_s)ds+G(Z_T,m_T)\Big],
\end{equation}
where $H^*(Z_s,\alpha_s)$ is the Legendre transform of $H(x,Xu)$ w.r.t. the second variable.

In this paper, we assume that the Hamiltonian is given by
\begin{equation}\label{eq1.22}
H(x,Xu)=\frac{1}{2}\left|Xu\right|^2,
\end{equation}
then the optimal feedback of each agent is given by
\begin{equation}\label{11.20}
\alpha^*(t,x)=-D_{p}H(x,Xu)=-Xu.
\end{equation}
For convenience, we use $\sigma^{ik}$ instead of $\sigma^{ik}(Z_s)$.
Note that the vanishing of the coefficient $\sigma$ implies that the typical player has some forbidden directions of states, here we assume that the matrix $\sigma(x)$ satisfies
\begin{equation}\label{con2}
 \sum^q_{k=1} \sum_{i=1}^{d}\sigma^{ik}(x)\partial_{x_i}\sigma^{jk}(x)=0,\ \text{for\ any}\ j=1,\ldots,d,
\end{equation}
a condition arising from relation \eqref{case11} below. 
\begin{Rem}
In the general degenerate MFG setting, the second order operator takes the form ${\rm{tr}}(\sigma\sigma^{T} D^2u)$ given by \eqref{2.2} as below. To further characterize this operator in terms of sum-of-squares operators $\Delta_\mathcal {X}$, we impose condition \eqref{con2} on the diffusion matrix \(\sigma(x)\). However, we emphasize that this condition \eqref{con2} arises natural, as it covers a broad range of settings, including operators on the Heisenberg group, Grushin-type operators, and many others. Concrete examples are provided as follows.
\end{Rem}
\begin{example}\label{Hes}
An easy example for the matrix $\sigma$ satisfying \eqref{con2} is
$$\sigma(x)=\left(
              \begin{array}{cc}
                1 & 0 \\
                0 & 1 \\
                -x_2 & x_1 \\
              \end{array}
            \right),\ \text{for}\ x=(x_1,x_2,x_3)\in\mathbb{R}^3.
$$
In view of \eqref{sigma}, the vector fields $\mathcal {X}=\{X_1,X_2\}$ is given by
$$X_1=\partial_{x_1}-x_2\partial_{x_3},\ X_2=\partial_{x_2}+x_1\partial_{x_3},$$
in the Heisenberg case. It is clear that the Kohn Laplacian
 $\triangle_\mathcal {X}=X^2_1+X^2_2$
  is a H\"{o}rmander operator 
and the set $\mathcal {X}$ is free up (see Definition \ref{FU}).
\end{example}
\begin{example}\label{Gru}
The Grushin-type operators given by
\[\mathcal {L}_G=\Delta_{x}+\lambda^2(x)\Delta_{y},\]
have been extensively studied \cite{F94}, where $\lambda(x)\in C(\mathbb{R}^n)$, $x\in\mathbb{R}^n,\ y\in\mathbb{R}^m$.

The corresponding diffusion matrix $\sigma(x)$ is an $(n+m)$-order diagonal square matrix, where the diagonal elements of the first $n$-dimensions are $1$, and the diagonal elements of the remaining $m$-dimensions are $\lambda(x)$. In view of \eqref{sigma}, the vector fields $\mathcal {X}=\{X_1,\ldots,X_n,X_{n+1},X_{n+m}\}$ is given by
$$X_i=\partial_{x_i},\ X_j=\lambda(x)\partial_{y_j},\ i=1,\ldots,n,\ j=n+1,\ldots,n+m.$$
The characteristic of Grushin-type operators the vector fields satisfy a symmetry property, that is, $X_k^*=-X_k$.
It is easy to verify that the coefficients $\sigma$ satisfy the condition \eqref{con2}.
\end{example}
Now we concisely give a derivation of the MFG system \eqref{MFG} as below.
Fixed $\bar{m}\in C([0,T],\mathcal {P}(\mathbb{T}^d))$, we first provide a derivation of the following HJE
\begin{equation}\label{HJE}
\left\{
 \begin{aligned}{ll}
 &-\partial_t u-\Delta_{\mathcal {X}} u+H(x,Xu)=F(x,\bar{m}),&(t,x)\in [0,T)\times\mathbb{T}^d,\\
  &u(T,x)=G(x,\bar{m}_T),& x\in\mathbb{T}^d.
 \end{aligned}
\right.
\end{equation}
The main idea  refers to Chapter 2 in \cite{R18}. For any stopping time $\tau\in[t,T]$, by using the It\^{o}'s formula to $u$ on $[t,\tau]$, we have
\begin{align}\label{2.2}
\begin{split}
u(\tau, Z_\tau)-u(t, Z_t)
&= \int^\tau_t\partial_s u+Du\sigma\alpha_s+{\rm{tr}}(\sigma\sigma^{T} D^2u)ds+\int^\tau_t Du\sigma dB_s.
\end{split}
\end{align}

Consider that the operator $X_k$ is defined in \eqref{sigma}, then we have
\begin{align}\label{case11}
\begin{split}
  \Delta_\mathcal {X} u= \sum^q_{k=1}\sum^d_{i=1}\sigma^{ik}\partial_{x_i}\big(\sum^d_{j=1}\sigma^{jk}\partial_{x_j}u\big) ={\rm{tr}}(\sigma\sigma^{T} D^2u)+\sum^q_{k=1}\sum^d_{i,j=1}\partial_{x_i}(\sigma^{jk})\sigma^{ik}\partial_{x_j}u.
\end{split}
\end{align}
 If the diffusion matrix $\sigma$ satisfies \eqref{con2}, then we have
\begin{equation*}
\Delta_{\mathcal {X}}u={\rm{tr}}(\sigma\sigma^{T} D^2u).
\end{equation*}
Since
$Du\sigma =Xu,$ then we have
\begin{equation}\label{ce2}
u(\tau, Z_\tau)-u(t, Z_t)
= \int^\tau_t\partial_s u+\Delta_{\mathcal {X}}u+Xu \alpha ds+\int^\tau_tXu dB_s.\end{equation}
By the dynamic programming principle, we have
\begin{equation}\label{1.08}
 u(t,x)=\inf_{\alpha\in\mathcal{A}}\mathbf{E}\Big[\int^\tau_t\frac{1}{2}|\alpha_s|^2
 +F(Z^{x,t}_s,\bar{m}_s)ds+u(\tau,Z^{x,t}_\tau)\Big],
\end{equation}
where $Z^{x,t}_\tau$ is the solution to SDE \eqref{SDE} starting from $x$ at the time $t$. Then martingale property gives that $\mathbf{E}\left[\int^\tau_t Xu dB_s\right]=0$. Plugging \eqref{ce2} into \eqref{1.08}, we have
\begin{align*}
   \inf_{\alpha\in\mathcal{A}}\mathbf{E}\Big[\int^{\tau}_{t}\frac{1}{2}|\alpha_s|^2+F(Z^{x,t}_s,\bar{m}_s)
   +\partial_su+\Delta_{\mathcal {X}}u+Xu \alpha ds\Big] &=  0.
\end{align*}
Let $\tau=t+\delta$, divide by $\delta$ and let $\delta\rightarrow0$, we obtain
\begin{equation}\label{H1}
  \partial_tu+\Delta_\mathcal {X} u+F(x,\bar{m})+\inf_{\alpha\in\mathcal{A}}\big\{H^*(x,\alpha)+Xu \alpha \big\}=0,
\end{equation}
that implies that the HJE \eqref{HJE} is valid.

Next we give the derivation of the following FPE
\begin{equation}\label{FPE}
\left\{
 \begin{aligned}{ll}
   & \partial_t m-\Delta_\mathcal {X} ^*m-{\rm{div}}_{\mathcal {X}^*}(mD_pH(x,Xu))=0,&(t,x)\in (0,T]\times\mathbb{T}^d,\\
   & m(0,x)=m_0,& x\in\mathbb{T}^d,
\end{aligned}
\right.
\end{equation}
which refers to Chapter 1 in \cite{C16}. If $\varphi:\mathbb{T}^d\rightarrow\mathbb{R}$ is a $C_{\mathcal {X}}^2$ function with bounded derivatives, using the It\^{o}'s formula for \eqref{FPE} on $[0,t]$, similar to \eqref{2.2}, we get
\begin{align}\label{2.3}
  \varphi(Z_t) 
   &= \varphi(Z_0)+\int^{t}_{0}\Delta_\mathcal {X}\varphi (Z_s)+X\varphi(Z_s) \alpha_sds+\sqrt{2}\int^{t}_{0}X\varphi(Z_s) dB_s.
\end{align}
We denote the distribution of $X_t$ by $\mu_t(dx)=P(Z_t\in dx)$, and use the notation
$\left<\varphi,\mu\right>=\int_{\mathbb{T}^d}\varphi(x)\mu(dx).$
Taking expectations on both sides of \eqref{2.3}, we have
\begin{equation}\label{mu}
\left <\varphi,\mu_t\right>=\big <\varphi,\mu_0\big>+\int^{t}_{0}\big <\Delta_\mathcal {X}\varphi+X\varphi \alpha,\mu_s\big>ds
                  =\big <\varphi,\mu_0+\int^{t}_{0}\Delta_\mathcal {X}^*\mu_s-{\rm{div}}_{\mathcal {X}^*}(\mu_s\alpha) ds\big>,
\end{equation}
the last equality is provided by the integration by parts.
Assume that $\mu_t$  has a density satisfying $\mu_t(dx)=m(t,x)dx$. For the arbitrary of $\varphi$ in \eqref{mu}, then the density $m$ is a solution of
\begin{equation}\label{eq2.11}
\partial_t m=\Delta_\mathcal {X}^*m-{\rm{div}}_{\mathcal {X}^*}(m\alpha),
\end{equation}
with the initial condition $m_0$. Combined with \eqref{11.20}, the FPE \eqref{FPE} is valid.

\subsection{Research history}
The study of degenerate equations  associated with  H\"{o}rmander  vector fields has attracted a lot of attention in recent decades. The theory of H\"{o}rmander operators has been advanced by the H\"{o}rmander's groundbreaking work \cite{Hormander}, in which he proposed the well-known H\"{o}rmander's theorem. To be specific, given a family of H\"{o}rmander vector fields, then the operator of the form
\begin{equation}\label{Lu}
\mathcal {L}=\sum^{q}_{i=1}X^2_i+X_0
\end{equation}
is hypoelliptic, which means that for any distribution $u$ in an open set $\Omega\subset\mathbb{R}^d$, $u$ is a $C^\infty$ function in every open set if $\mathcal {L}u$ is a $C^\infty$ function there. Moreover, the parabolic counterpart of \eqref{Lu} with $X_0=\partial_t$ is also
hypoelliptic. In recent years, more general classes of linear operators with variable coefficients structured on H\"{o}rmander vector fields, namely
\[\mathcal {L}=\sum^q_{i,j=1}a_{ij}(x)X_iX_j+a_0(x)X_0+\sum^q_{i=1}b_i(x)X_i+c(x),\]
have also come under study, we can refer to \cite{a0,Bramanti}.

Existing literature contains extensive theoretical studies on H\"{o}rmander operators, and  a priori Schauder estimates related to our work have been widely investigated by scholars.
Roughly speaking, Schauder estimates state that, given a solution $u$ to the inhomogeneous equation \(\mathcal{H}u = f\), if both the coefficients of the operator \(\mathcal{H}\) and the inhomogeneous term $f$ are H\"{o}lder continuous, then this regularity is propagated by the operator, thereby endowing the derivatives of the solution with H\"{o}lder continuity.

Notably, the global theory  is inevitably missing for the general H\"{o}rmander operators, since one cannot expect that the operators can be equipped with a global fundamental solution defined out of the diagonal of $\mathbb{R}^d\times\mathbb{R}^d$ , without further assumptions on operators.
Xu \cite{Xu90, Xu92} established a priori interior estimates for subelliptic operators. Based on properties of
the ``heat kernel", Bramanti and Brandolini \cite{Bramanti} established local a priori Schauder estimates for the parabolic counterpart.  

In some special cases, the local Schauder estimates can be derived to global Schauder estimates. For the homogeneous left invariant H\"{o}rmander operators on homogeneous groups, Folland \cite{Folland} developed the global theory.
Furthermore, within the framework of Carnot groups, owing to the fact that H\"{o}rmander operators possess translation invariance and homogeneous dilation invariance, we refer to 
Bonfiglioli et al. \cite{BL04}, Guti\'{e}rrez and Lanconelli \cite{GL} 
and references therein.

Unlike the existing works, this paper focuses on a class of H\"{o}rmander operators without presupposing the existence of any group structure, and we provide an alternative proof method for a priori Schauder estimates for degenerate parabolic equations via Campanato spaces, which is not based on the representation of fundamental solutions.

MFG theory is devoted to the analysis of differential games with infinitely many players. This theory has been introduced by Lasry and Lions in 2006 \cite{L06,LL06}. At about the same time, Huang et al. \cite{HM06} solved the large population games independently. Then MFG has been studied extensively in many different fields. Bensoussan et al. \cite{BF13} studied the MFG and mean field type control theory. Carmona and Delarue \cite{CD13} focused on the theory and applications of MFG by probabilistic approach. Gangbo \cite{GS14} developed optimal transport theory within the MFG framework.  Gomes et al. \cite{GP16} discussed regularity theory for MFG system either stationary or time-dependent, local or nonlocal.  Cardaliaguet et al. \cite{CD19} obtained  the existence of classical solutions for the master equation of MFG. The notes written by Cardaliaguet \cite{C12} and by Ryzhik \cite{R18} showed the more comprehensive analysis of the MFG.

Degenerate MFG systems are more common than classical MFG systems, scholars studied the suitably defined weak solutions for general form of degenerate MFG systems as follows
  \begin{equation}\label{mfg1}
\left\{
  \begin{aligned}
    &-\partial_t u-{\rm{tr}}(\sigma\sigma^{T}(x)D^2 u)+H(x,Du)=F(x,m), & (t,x)\in [0,T)\times\mathbb{R}^d, \\
    &\partial_t m-{\rm{div}}\big({\rm{div}}(\sigma\sigma^{T}(x)m)+mD_pH(x,Du)\big)=0,& (t,x)\in(0,T]\times\mathbb{R}^d, \\
   & u(T,x)=G(x,m_T),\  m(0,x)=m_0, &x\in\mathbb{R}^d.
  \end{aligned}
\right.
\end{equation}
 Cardaliaguet et al. \cite{CG15} obtained the existence and uniqueness of weak solutions for the MFG systems  with degenerate diffusion and local coupling. Then Ferreira et al. \cite{FG21} extended the results to a wide class of time-dependent degenerate systems. 
In order to investigate more regularity beyond weak solutions, the hypoelliptic operator is introduced into degenerate MFG systems.

The hypoelliptic MFG was addressed by Dragoni and Feleqi \cite{DF18}. They studied an ergodic system with H\"{o}rmander diffusions  as follows
\begin{equation*}
\left\{
  \begin{aligned}{ll}
    &-\Delta_\mathcal {X} u+\lambda u+H(x,Xu)=F(m), & x\in\mathbb{T}^d, \\
    &-\Delta_\mathcal {X} ^*m-{\rm{div}}_{\mathcal {X}^*}(mD_pH(x,Xu))=0,& x\in
    \mathbb{T}^d,\\
    & \int_{\mathbb{T}^d}mdx=1, \ m>0,
     \end{aligned}
\right.
\end{equation*}
and obtained a unique classical solution in the intrinsic spaces.
Their interior regularity for subelliptic equation given by  Xu and Zuily \cite{XZ97} and references therein.
Then Feleqi et al. \cite{FG20} showed regularities for the hypoelliptic MFG system based on the theory of eigenvalue problems.

The degenerate parabolic  MFG systems of H\"{o}rmander type, despite its great relevance from the applicative viewpoint, has been less investigated. Mannucci et al. \cite{MMM24}  obtained the classical solutions to the following MFG system
 \begin{equation*}
\left\{
  \begin{aligned}
    &-\partial_t u-\lambda\Delta_\mathcal {X} u+H(x,Xu)=F(x,m), & (t,x)\in[0,T)\times\mathbb{R}^d, \\
    &\partial_t m-\lambda\Delta_\mathcal {X}m-{\rm{div}}_{\mathcal {X}}(mD_pH(x,Xu))=0,& (t,x)\in(0,T]\times\mathbb{R}^d, \\
    & u(T,x)=G(x,m_T),\  m(0,x)=m_0, &x\in\mathbb{R}^d,
  \end{aligned}
\right.
\end{equation*}
 with superlinear growth of the Hamiltonian defined on an homogeneous Lie group, $\lambda>0$.
Their result provided the regularity of solutions, which relies on the homogeneity of the vector fields and the representation of the fundamental solutions. Some available results on degenerate parabolic equations can refer to \cite{FL95,BL04,BB10}.

To the best of our knowledge, for parabolic MFG systems with H\"{o}rmander diffusions,there exist no directly applicable results for the associated degenerate parabolic equations without recourse to fundamental solutions. In particular, the corresponding a priori estimates are absent in the existing literature.

\subsection{ Main assumptions and main results}
Throughout this paper, $C$ is a generic positive constant which may take different values in different contexts.  The main results proved in this paper are divided into two aspects. The first aspect is related to the degenerate parabolic equation \eqref{peq} associated with vector fields $\mathcal {X}$ given in \eqref{sigma}.
Now we make the hypotheses for the coefficients of the vector fields $\mathcal {X}$ as follows:

\noindent{\bf{(H1)}} The matrix-valued coefficient $\sigma(x)$ is allowed to be rank-degenerate but non-vanishing, and all elements $\{\sigma^{ik}(x)\}$ are smooth function for any $i\in\{1,\ldots,d\}$, $k\in\{1,\ldots,q\}$ and $x\in\mathbb{T}^d$, such that
the vector fields $\mathcal {X}$ satisfy the H\"{o}rmander's condition.

Without loss of generality we assume that these vector fields $\mathcal {X}$ are free up to the order $q$ (see Definition \ref{FU}).

First, we give  a priori global Schauder estimates for the degenerate parabolic equation \eqref{peq} on $\mathbb{T}^d$ induced by the general H\"{o}rmander vector fields as follows.
\begin{Prop}\label{prop13} Given the vector fields $\mathcal {X}$ defined in \eqref{sigma} and satisfying the hypotheses {\bf{(H1)}}, assume that the coefficients  $c,f\in C^{0,\alpha}_{\mathcal {X}}\big((0,T]\times\mathbb{T}^d\big)$, $g\in C^{2,\alpha}_{\mathcal {X}}\big(\mathbb{T}^d\big)$. If $u\in C_{\mathcal {X}}^{2,\alpha}([0,T]\times\mathbb{T}^d)$ is a solution of the equation \eqref{peq}, then there exists a positive constant $C$ such that
\begin{equation}\label{3.35}
\|u\|_{C^{2,\alpha}_{\mathcal {X}}([0,T]\times\mathbb{T}^d)}\leq C\big(\|u\|_{L^{\infty}([0,T]\times\mathbb{T}^d)}+\|f\|_{C^{0,\alpha}_{\mathcal {X}}([0,T]\times\mathbb{T}^d)}+\|g\|_{C^{2,\alpha}_{\mathcal {X}}(\mathbb{T}^d)}\big).
\end{equation}
\end{Prop}
Here we emphasize that a priori Schauder estimates above are established via Campanato spaces, instead of employing the fundamental solution framework.
\begin{Rem}
If we consider the equation \eqref{peq} posed on \([0,T]\times\mathbb R^d\), without assuming periodicity for the coefficients, then the global Schauder estimates cannot be obtained. This arises from the absence of global fundamental solutions for general H\"{o}rmander  operators without any group structure. Consequently, we are only able to derive the local Schauder estimates given by \eqref{aij}, which appear in the proof of Proposition \ref{prop13}, that is, for  any $\Omega'\Subset \Omega$, $\Omega$ is a open subset of $\mathbb{R}^d$,
\begin{equation}\label{locales}
 \|u\|_{C^{2,\alpha}_{\mathcal {X}} ([0,T]\times\Omega')}\leq C\big(\|u\|_{L^{\infty}([0,T]\times\Omega)}+\|f\|_{C^{0,\alpha}_{\mathcal {X}} ([0,T]\times\Omega)}\big),
\end{equation}
where $C$ is a constant depending on $\Omega,\Omega',\mathcal {X},\alpha$ and $C^{0,\alpha}_{\mathcal {X}}$ norms of the coefficients.

 Furthermore, if we assume that all coefficients of H\"{o}rmander vector fields and the equation \eqref{peq} are periodic, then we obtain  the global Schauder estimates as follows
 \begin{equation}\label{globalrd}
 \|u\|_{C^{2,\alpha}_{\mathcal {X}} ([0,T]\times\mathbb{R}^d)}\leq C\big(\|u\|_{L^{\infty}([0,T]\times\mathbb{R}^d)}+\|f\|_{C^{0,\alpha}_{\mathcal {X}} ([0,T]\times\mathbb{R}^d)}\big).
\end{equation}

Deriving the global estimate \eqref{globalrd} from the local estimate \eqref{locales} requires the periodicity of solutions $u$ given by Lemma \ref{periodic} and the translation invariance of the distance $d_c$ given by the claim \eqref{dk}.
\end{Rem}
\begin{Rem}
To obtain the global Schauder estimates \eqref{3.35} on $[0,T]\times\mathbb{T}^d$, we first establish the global  estimate \eqref{globalrd} under the assumption that equation \eqref{peq} is defined on \([0,T]\times\mathbb R^d\) with all coefficients periodic. The derivation from \eqref{globalrd} to \eqref{3.35} also relies on the periodicity of solutions $u$ given by Lemma \ref{periodic} and the translation invariance of the distance $d_c$ given by the claim \eqref{dk}.
\end{Rem}
\begin{Rem}
Note that condition \eqref{con2} satisfied by the diffusion matrix \(\sigma(x)\) is only a condition relating the SDE \eqref{SDE} to the MFG system \eqref{MFG} characterized by the sum-of-squares operator, and is not imposed as an assumption for analyzing the wellposedness of the MFG system \eqref{MFG}.
\end{Rem}
Next, we obtain the existence and uniqueness of the equation  \eqref{peq} in intrinsic H\"{o}lder space. The proof is a blend of the the continuity method and modification method.
\begin{Thm}\label{eu}
Under the same assumptions of Proposition \ref{prop13}, there exists a unique solution $u$ of the equation \eqref{peq}.
 Moreover, $u\in C^{2,\alpha}_{\mathcal {X}}([0,T]\times\mathbb{T}^d)$, and satisfies
\begin{equation}\label{21.2}
  \|u\|_{C^{2,\alpha}_{\mathcal {X}}([0,T]\times\mathbb{T}^d)}\leq C.
\end{equation}
\end{Thm}

The second aspect concerns the MFG systems \eqref{MFG} with H\"{o}rmander diffusions. We obtain the well-posedness of the degenerate MFG systems \eqref{MFG} by using the regularity properties established above. Now we list the notations for MFG theory as follows.

Let $\mathcal {P}(\mathbb{T}^d)$ be the set of Borel probability measures $m$ on $\mathbb{T}^d$ such that $\int_{\mathbb{T}^d}|x|dm(x)<\infty$, and endowed with the Kantorovitch-Rubinstein distance
\begin{equation}\label{distance}
 d_1(\mu,\nu):=\inf_{\gamma\in\Pi(\mu,\nu)}\int_{\mathbb{T}^{2d}}|x-y|d\gamma(x,y),
\end{equation}
where $\Pi(\mu,\nu)$ is the set of Borel probability measures on $\mathbb{T}^{2d}$ such that $\gamma(E\times\mathbb{T}^d)=\mu(E)$,
 and $\gamma(\mathbb{T}^d\times E)=\nu(E)$ for any Borel set $E\subset\mathbb{T}^d$.

 Let $\mathcal {C}$ be the set of maps $\mu\in C([0,T];\mathcal {P}(\mathbb{T}^d))$ such that $\sup\limits_{t\in[0,T]}\int_{\mathbb{T}^d}|x|^2d\mu(x)\leq C$,
\begin{equation}\label{eq297}
\sup_{s\neq t}\big\{|t-s|^{-\frac{1}{2}}d_1(\mu_s,\mu_t)\big\}\leq C.
\end{equation}
Then $\mathcal {C}$ is a convex closed subset of $C([0,T];\mathcal {P}(\mathbb{T}^d))$. In fact, $\mathcal {C}$ is compact in $\mathcal {P}(\mathbb{T}^d)$. More properties of the distance $d_1$ and the space of measures see Chapter 5 in  \cite{C12}.

The standard hypotheses in intrinsic H\"{o}lder spaces are required as follows.

\noindent{\bf(H2)}  The probability measure $m_0$ is absolutely continuous with respect to Lebesgue measure, and has a H\"{o}lder continuous density (still denoted $m_0$), which satisfies $\int_{\mathbb{T}^d}|x|^2 dm_0<\infty$.

\noindent{\bf(H3)} The functions $F(x,m)$ and $G(x,m_T)$ are Lipschitz continuous in $\mathbb{T}^d\times\mathcal {P}(\mathbb{T}^d)$, that is for any $(x,\mu),(y,\nu)\in\mathbb{T}^d\times\mathcal {P}(\mathbb{T}^d)$,

{\rm{(H3.1)}} $|F(x,\mu)-F(y,\nu)|\leq C \big(d_c(x,y)+d_1(\mu,\nu)\big)$;

{\rm{(H3.2)}} $|G(x,\mu_T)-G(y,\nu_T)|\leq C\big(d_c(x,y)+d_1(\mu_T,\nu_T)\big)$,\\
where $d_c$ is given in Definition \ref{2.6} below.


\noindent{\bf(H4)} The functions $F(x,m)$ and $G(x,m_T)$ are monotonically increasing with respect to  the measure $m$, that is for any $\mu,\nu\in\mathcal {P}(\mathbb{T}^d)$,

{\rm{(H4.1)}} $\int_{\mathbb{T}^d}\left(F(x,\mu)-F(x,\nu)\right)d(\mu-\nu)(x)\geq 0$, $\mu\neq\nu$;

{\rm{(H4.2)}} $\int_{\mathbb{T}^d}\left(G(x,\mu_T)-G(x,\nu_T)\right)d(\mu_T-\nu_T)(x)\geq 0$.

\begin{Rem}
All hypotheses above for the degenerate MFG system \eqref{MFG} are consistent with those in the classical MFG systems, while the corresponding distances and spaces are induced by H\"{o}rmander vector fields $\mathcal {X}$. For standard hypotheses of classical case, one can refer to \cite{C12} .
\end{Rem}

The MFG system \eqref{MFG} consists of the HJE \eqref{HJE} with  the quadratic Hamiltonian and the FPE \eqref{FPE} with the dual operator.
 For the quasi-linear HJE \eqref{HJE}, we use the Hopf-Cole transform, setting $$w(t,x)=e^{-\frac{u(T-t,x)}{2}}, \quad \text{for\ any\ }(t,x)\in[0,T]\times\mathbb{T}^d.$$
It is clear that $u$ is a solution for the quasi-linear equation  \eqref{HJE},
 if and only if $w$ is a positive solution for the following linear equation
\begin{equation}\label{2.0}
\left\{
  \begin{aligned}
   & \partial_t w-\Delta_{\mathcal {X}}w+\frac{1}{2}F w=0,&(t,x)\in (0,T]\times\mathbb{T}^d, \\
   & w(0,x)=e^{-\frac{G(x,\bar{m}_T)}{2}},& x\in\mathbb{T}^d.
  \end{aligned}
\right.
\end{equation}
Here we ingeniously combine the weak maximum principle to give a positive lower bound for  solutions $w$, which constitutes a crucial ingredient for applying  the existence results of the linear equation \eqref{peq} to quasi-linear equations. 

As a consequence of Theorem \ref{eu}, we can obtain the existence and uniqueness of the solution to the HJE \eqref{HJE} in the intrinsic H\"{o}lder space.
\begin{Prop}\label{HJEPF}
Under assumptions {\bf{(H1)}}-{\bf{(H4)}}, there exists a unique solution $u$ of the equation \eqref{HJE}.  Moreover, $u\in C^{2,\alpha}_{\mathcal {X}}([0,T]\times\mathbb{T}^d)$, and satisfies
\begin{equation}\label{uhje}
  \|u\|_{C^{2,\alpha}_{\mathcal {X}}([0,T]\times\mathbb{T}^d)}\leq C.
\end{equation}
\end{Prop}

For the FPE \eqref{FPE}, we need to consider the existence and uniqueness of solutions to its dual equation as follows
\begin{equation}\label{DFPE}
\left\{
 \begin{aligned}{ll}
   &-\partial_t m-\Delta_{\mathcal {X}}m-XmXu=\phi,&(t,x)\in (0,T]\times\mathbb{T}^d,\\
   & m(T,x)=0,& x\in\mathbb{T}^d,
\end{aligned}
\right.
\end{equation}
for any $\phi\in L^2([0,T]\times\mathbb{T}^d)$.
 Then by using a variant of Lax-Milgram theorem and the duality method, we obtain the existence and the  uniqueness  of solutions to the FPE \eqref{FPE}.

\begin{Prop}\label{3.2}
Under the assumption {\bf(H1)}, the FPE \eqref{FPE} has a unique solution $m\in H^{1}_{\mathcal {X}}([0,T]\times\mathbb{T}^d)$ in the sense of distribution.
\end{Prop}

Combining all preceding results, we achieve our main goal by Schauder fixed point theorem. Specifically, we obtain the existence and uniqueness of the coupling solution for the degenerate MFG system \eqref{MFG} in the intrinsic H\"{o}lder space as follows.
\begin{Thm}\label{thm3}
 Under the assumptions {\bf{(H1)}}-{\bf{(H4)}}, given the vector fields $\mathcal {X}$ defined in \eqref{sigma}, there exists a unique coupling solution $(u,m)\in C^{2,\alpha}_{\mathcal {X}}([0,T]\times\mathbb{T}^d)\times C([0,T];\mathcal {P}(\mathbb{T}^d))$ to the degenerate MFG system \eqref{MFG}.
\end{Thm}

In this paper, we focus on the well-posedness and regularity of coupling solutions for a class of general degenerate MFG systems. By introducing the H\"{o}rmander vector fields with the natural condition \eqref{con2}, we obtain the existence and uniqueness of the solutions in intrinsic H\"{o}lder spaces for the MFG system \eqref{MFG} given by Theorem \ref{thm3}. From the partial differential equations (PDE, for short) perspective, we establish global theory for degenerate parabolic equations  \eqref{peq} on the torus \(\mathbb{T}^d\) given by Proposition \ref{prop13} for the first time, here the  H\"{o}rmander operators are not required to possess any group structure. 
Detailed obstacles are outlined in the following four parts.

 First, we establish a rigorous characterization of general degenerate MFG systems by means of general H\"{o}rmander  sum-of-squares operators satisfying the condition \eqref{con2}. We emphasize that this condition \eqref{con2} arises natural, as it covers a broad range of settings, including operators on the Heisenberg group, Grushin-type operators, and many others. Concrete examples are provided in Example \ref{Hes} and Example \ref{Gru}.  From the perspective of a single agent, the MFG systems \eqref{MFG} with H\"{o}rmander diffusions mean that each agent can move only along admissible direction at some point, thereby covering a wider class of applications including the classical setting.

 Second, we prove the well-posedness of the MFG systems \eqref{MFG} via Schauder fixed point theorem, its coupled solutions  can describe the Nash equilibria for a differential game with infinitely many small players. To this end, we need to obtain the existence and uniqueness results for the HJE \eqref{HJE} and the FPE \eqref{FPE} individually.
 For the FPE \eqref{FPE}, we apply a variant of Lax-Milgram theorem to obtain existence and uniqueness of weak solutions for its dual problem, from which we deduce the well-posedness of the FPE \eqref{FPE} via the duality method.
For the HJE \eqref{HJE}, using the Hopf transform together with the weak maximum principle, we show that the existence and uniqueness of the solutions in intrinsic H\"{o}lder spaces to the quasilinear problem \eqref{HJE} follows from that of its linearized counterpart, as established in Theorem \ref{eu}.

Third, for the degenerate parabolic equations \eqref{peq} defined on $\mathbb{T}^d$ associated with general H\"{o}rmander vector fields, we develop a global regularity theory for the first time within this framework. Unlike the existing works, we focus on H\"{o}rmander operators without any group structure. Here we provide an alternative proof method for a priori Schauder estimates for degenerate parabolic equations via Campanato spaces, which replaces the usual proofs based on the representation of fundamental solutions. The difficulty in our case is the lack of  commutation and the lack of homogeneity of the vector fields. The first one prevents standard derivatives to satisfy the original equation. Instead, we construct a polynomial corrector with respect to the solution, which yields an interior estimate for the difference between the derivative and its mean value, given Proposition \ref{X2uuR}. Note that the derivation of the Poincar\'{e}-type inequality for a parabolic version also constitutes a vital ingredient in the above proof.

Fourth, the global Schauder estimate on $\mathbb{T}^d$ stated in Proposition \ref{prop13} is guaranteed by both the corresponding global estimates on \(\mathbb{R}^d\) and the periodicity of the coefficients. Due to the lack of weak maximum principle on general H\"{o}rmander operators, here we use the Feynman-Kac formula to obtain the periodicity of solutions, instead of the PDE method. The Lipschitz condition of the coefficients $\sigma$ guarantees the pathwise uniqueness of solution.

The rest of the paper is organized as follows. In Section 2, we set up the intrinsic spaces induced by H\"{o}rmander vector fields, and give some known results. In Section 3, for the degenerate parabolic equations, we obtain the Schauder estimates given in Proposition \ref{prop13} and the regularity of solutions given in Theorem \ref{eu}. In Section 4, we give the existence and uniqueness of solutions to the HJE \eqref{HJE} given by Proposition \ref{HJEPF}, and to the FPE \eqref{FPE} given by Proposition \ref{3.2}. Then by the Schauder fixed point theorem, we obtain the well-posedness to the degenerate MFG system \eqref{MFG} given in Theorem \ref{thm3}.
\section{PRELIMINARIES AND KNOWN RESULTS}
\subsection{Intrinsic spaces induced by the H\"{o}rmander vector fields}
 Let $\mathcal {X}=\{X_1,\ldots,X_q\}$, $2\leq q< d$, a family of real smooth vector fields which are defined in some open set $\Omega\subset\mathbb{R}^d$.  Let us assign to each $X_i$ a weight $p_i=1$,  for any $i=1,2,\ldots,q.$
 For any multi-index $I=(i_1,i_2\ldots,i_s)$, $1\leq i_j\leq q$, we define the weight of $I$ is $|I|=s.$
 For any couple of vector fields $X$ and $Y$, we define their commutator as
$$[X,Y]:=XY-YX.$$
Set $X^I :=X_{i_1}X_{i_2}\ldots X_{i_s}$ and
$$X^{[I]} :=[X_{i_s},[X_{i_{s-1}},\ldots[X_{i_3},[X_{i_2},X_{i_1}]]\ldots]].$$
We will say that $X^{[I]}$ is a commutator with weight $|I|$.
Now, let us consider the Lie algebra generated by the vector fields $\mathcal {X}$, that is the set of vector fields obtained as linear combination of the fields $\mathcal {X}$ and their commutators of any finite step.
\begin{Def}\label{HC}{\rm{(H\"{o}rmander's condition)}}
Assume that a system of vector fields $\mathcal {X}$ have real smooth coefficients and the Lie algebra generated by the vector fields $\mathcal {X}$ span
$\mathbb{R}^d$ at every point of  an open set $\Omega\subset\mathbb{R}^d$. Then we call that the smooth vector fields $\mathcal {X}$ satisfy the ``H\"{o}rmander's condition" in some domain $\Omega$, and we also call $\mathcal {X}$ as ``H\"{o}rmander vector fields".
\end{Def}
Furthermore, we say that H\"{o}rmander vector fields $\mathcal {X}$ with the H\"{o}rmander index $\rm{k}$ if  these vector fields, together with their commutators of weight $|I|\leq\rm{k}\in\mathbb{N}$, span the tangent space at every point of $\Omega$.

\begin{Def} \label{FU}{\rm{(Free up to step $ r $)}} We say that a system of smooth vector fields $\mathcal {X}=\{X_1,\ldots,X_q\}$ is free
up to step $ r $ in a domain $\Omega$ of $\mathbb{R}^d$ if $\mathcal {X}$ and their commutators up to step $ r $ do not satisfy any linear relation other than those which hold automatically as a consequence of antisymmetry of the Lie bracket and Jacobi identity.
\end{Def}
For more details on H\"{o}rmander  vector fields see in \cite{Hormander,B14}.
Recalling some basic definitions induced by a family of vector fields $\mathcal {X}$ as follows.
\begin{Def}\label{2.6}{\rm{(Carnot-Carath\'{e}odory distance)}}
For any $x,\ y\in\Omega$, let
\begin{equation}\label{d1}
d_{c}(x,y)=\inf\{l(\gamma)|\ \gamma:[0,l(\gamma)]\rightarrow \Omega\ \text{is}\ \mathcal{X}\text{-subunit},\ \gamma(0)=x,\ \gamma(l(\gamma))=y\},
\end{equation}
where we call $\mathcal{X}$-subunit any absolutely continuous curve $\gamma$ such that
\begin{equation}\label{gamma}
\gamma'(t)=\sum^{q}_{j=1}\lambda_j(t)X_j(\gamma(t)),\ \text{a.e.,\ with}\ \sum^{q}_{j=1}\lambda^2_j(t)\leq1\ a.e..
\end{equation}
Then $d_{c}$ is a distance called the Carnot-Carath\'{e}odory distance, or CC-distance.
\end{Def}

\begin{Rem}
According to Chow's theorem \cite{Nagel85}, if the vector fields satisfy the H\"{o}rmander's condition, the set defined in \eqref{d1} is nonempty, so that $d_{c}$ is finite and continuous with respect to  the original Euclidean topology included on $\mathbb{R}^d$, see in \cite{Montgomery}.
\end{Rem}

A known result in \cite{Fefferman} states that, for any small compact subset $\Omega'\subset\Omega$, there exists a constant $C>0$ such that
\begin{equation}\label{2.7}
C^{-1}|x-y|\leq d_c(x,y)\leq C|x-y|^{\frac{1}{{\rm k}}},
\end{equation}
for any $x,y\in \Omega'$, where $r$ is the H\"{o}rmander index. Then we define a family of balls
 \[B_R(x)=\{y\in \Omega:d_c(x,y)<R\},\]
for any $x\in\Omega'$ and $R>0$ small enough. There exists some $Q>0$, called the homogeneous dimension, such that for any $R>0$ small enough,
\begin{equation}\label{2.8}
  C^{-1}R^{Q}\leq|B_{R}(x)|\leq CR^Q.
\end{equation}
For more details on $CC$-distance see in \cite{Nagel85}. 
Now let us introduce the parabolic Carnot-Carath\'{e}odory distance $d_P$ corresponding to $d_{c}$ in $\mathbb{R}^{d+1}$, namely,
\begin{equation}\label{929}
d_P((t,x),(s,y))=\sqrt{d^2_{c}(x,y)+|t-s|},
\end{equation}
and the corresponding ball is given by
$$
B_P((t,x),R)=\big\{(s,y)\in\mathbb{R}^{d+1}:d_P((t,x),(s,y))<R\big\}.
$$
It's easy to prove $d_P((t,x),(s,y))$ is a distance on $\mathbb{R}\times \Omega$, see Lemma 3.4 in \cite{Bramanti}.

 Let $\Omega_T:=(0,T)\times \Omega\subseteq \mathbb{R}^{d+1}$ be an open set, for any $\alpha\in(0,1)$, the seminorm of  H\"{o}lder space is defined as
\begin{equation}\label{8.05}
[u]_{C^{0,\alpha}_{\mathcal {X}}(\Omega_T)}:=\sup\left\{\frac{|u(t,x)-u(s,y)|}{d^{\alpha}_{P}((t,x),(s,y))}:(t,x),(s,y)\in \Omega_T,(t,x)\neq(s,y)\right\},
\end{equation}
the norm of intrinsic H\"{o}lder space is defined as
$$\|u\|_{C^{0,\alpha}_{\mathcal {X}}(\Omega_T)}:=[u]_{C^{0,\alpha}_{\mathcal {X}}(\Omega_T)}+\|u\|_{L^{\infty}(\Omega_T)}.$$
\begin{Def}\label{2.06}{\rm{(Intrinsic H\"{o}lder space)}}
Let $r$ be a nonnegative integer, the set
$$C^{r,\alpha}_{\mathcal {X}}(\Omega_T):=\{u:\Omega_T\rightarrow\mathbb{R},\ \text{such\ that}\ \|u\|_{C^{r,\alpha}_{\mathcal {X}}(\Omega_T)}<\infty\},$$
is called a intrinsic H\"{o}lder space related to H\"{o}rmander vector fields $\mathcal {X}$ with the norm
\begin{equation}\label{21.9}
  \|u\|_{C^{r,\alpha}_{\mathcal {X}}(\Omega_T)}:=\sum_{|I|+2j\leq r}\|\partial_t^{j}X^{I}u\|_{{C^{0,\alpha}_{\mathcal {X}}(\Omega_T)}},
\end{equation}
where $I$ is multi-index, $j$ is a nonnegative integer.
\end{Def}
\begin{Def}\label{2.07}{\rm{(Intrinsic Sobolev space)}}
For $1\leq p<\infty$, the set
$$W^{r,p}_{\mathcal {X}}(\Omega_T):=\big\{u\in L^p (\Omega_T):\partial^{j}_{t}X^I u\in L^p(\Omega_T),\ \text{for\ any\ } I\ \text{satisfies}\ |I|+2j\leq r\big\},$$
is called a intrinsic Sobolev space related to H\"{o}rmander vector fields $\mathcal {X}$ with the norm
$$\|u\|_{W^{r,p}_{\mathcal {X}}(\Omega_T)}:=\Big(\sum_{|I|+2j\leq r}\iint_{\Omega_T}|\partial^{j}_{t}X^I u|^p dxdt\Big)^{\frac{1}{p}}.$$
For $p=2$, we write $H^{r}_{\mathcal {X}}(\Omega_T)$ instead of $W^{r,p}_{\mathcal {X}}(\Omega_T)$.
\end{Def}
\begin{Def}\label{2.08}{\rm{(Campanato space)}} Let $p\geq1$ and $\lambda\geq0$. A function $u\in L^p_{loc}(\Omega_T)$ is said to belong to the Campanato space $\mathcal {L}^{p,\lambda}_{\mathcal {X}}(\Omega_T)$ if
$$\|u\|_{\mathcal {L}^{p,\lambda}_{\mathcal {X}}(\Omega_T)}=[u]_{p,\lambda,\Omega_T}+\|u\|_{L^p(\Omega_T)}<\infty,$$
where $$[u]_{p,\lambda,\Omega_T}=\sup_{(t,x)\in\Omega_T,R\in(0,d_0)}\Big(R^{-\lambda}\iint_{\Omega_T\cap B_P((t,x),R)}|u-(u)_{P}|^{p}dxdt\Big)^{\frac{1}{p}},$$
and $(u)_{P}$ is the average of $u$ on the parabolic ball centered at $(t,x)$ with radius $R$, $d_0=\min\{diam(\Omega_T),R_0\}$, for fixed $R_0$.
\end{Def}
\begin{Rem}
For any $1\leq p<\infty$, the following embeddings hold
\begin{equation}\label{1019}
  C^{\mathrm{k}{r} ,\alpha}_{\mathcal {X}}(\Omega)\hookrightarrow C^{r,\frac{\alpha}{\mathrm{{k}}}}(\Omega),\quad W^{\mathrm{k}{r},p}_{\mathcal {X}}(\Omega)\hookrightarrow W^{r,p}(\Omega),
\end{equation}
where $\mathrm{k}$ is the H\"{o}rmander index.
The first is proved in \cite{Xu90} and the second is proved in \cite{Xu92}.
More properties about continuity, compactness and embedding in the intrinsic spaces, see in \cite{BB10}.
\end{Rem}


\subsection{Known results of intrinsic space}
\begin{Lem}\label{Sch}{\rm{(Theorem 1.1 in \cite{Bramanti})}}
Let $\Omega_T$ be a bounded subset of $\mathbb{R}\times\Omega$. Then for every domain $\Omega'_T\Subset \Omega_T$, there exists a constant $C(\Omega_T,\Omega'_T,\mathcal {X},\alpha)>0$, $\alpha\in(0,1)$, such that for each $u \in C^{2,\alpha}_{\mathcal {X}}(\Omega_T)$ with $\mathcal {H}u\in C_{\mathcal {X}}^{0,\alpha}(\Omega_T)$ one has
\begin{equation}\label{Sch1}
  \|u\|_{C_{\mathcal {X}}^{2,\alpha}(\Omega'_T)}\leq C\big(\|\mathcal {H}u\|_{C_{\mathcal {X}}^{0,\alpha}(\Omega_T)}+\|u\|_{L^{\infty}(\Omega_T)}\big),
\end{equation}
where $\mathcal {H}$ is given by \eqref{001}.
\end{Lem}

Then we give a parabolic version of the equivalence property.
 \begin{Lem}\label{A3}{\rm{(Lemma 2.3 in \cite{Y14})}}
The Campanato spaces $\mathcal {L}^{2,Q+2+2\alpha}_{\mathcal {X}}(\Omega_T)$ given in definition \ref{2.08} and the intrinsic Sobolev space $C^{0,\alpha}_{\mathcal {X}}(\Omega_T)$ given in Definition \ref{2.07}, $0<\alpha<1$, are topologically and algebraically isomorphic. Here $Q$ is the homogenous dimension with respect to  a set of vector fields $\mathcal {X}$.
\end{Lem}
Next, we give the interpolation inequality in intrinsic H\"{o}lder space.
\begin{Lem}\label{IE}{\rm{(Proposition 2.2 in \cite{Xu92})}}
Suppose $j+\beta<k+\alpha$, $j,k\in\mathbb{N}$, $\alpha,\beta\in[0,1]$, and $u\in C^{k,\alpha}_\mathcal {X}(\Omega)$. Then for any $\varepsilon>0$, there exists a constant $C(\varepsilon,j,k,\Omega)$ such that
\begin{equation}\label{927}
 \|u\|_{C^{j,\beta}_\mathcal {X}(\Omega)}\leq \varepsilon\|u\|_{C^{k,\alpha}_\mathcal {X}(\Omega)}+C\|u\|_{L^{\infty}(\Omega)}.
\end{equation}
\end{Lem}
Furthermore, we give a parabolic version of the interpolation inequality as follows.
\begin{Cor} \label{P}
Suppose $j+\beta<k+\alpha$, $j,k\in\mathbb{N}$, $\alpha,\beta\in[0,1]$, and $u\in C^{k,\alpha}_\mathcal {X}(\Omega_T)$. Then for any $\varepsilon>0$, there exists a constant $C(\varepsilon,j,k,\Omega_T)$ such that
\begin{equation}\label{928}
 \|u\|_{C^{j,\beta}_\mathcal {X}(\Omega_T)}\leq \varepsilon\|u\|_{C^{k,\alpha}_\mathcal {X}(\Omega_T)}+C\|u\|_{L^{\infty}(\Omega_T)}.
\end{equation}
\end{Cor}
\begin{proof}
By the definition of $d_P$ given in \eqref{929}, for any $(t,x),(s,y)\in\Omega_T$, we have
 $$\frac{|u(t,x)-u(s,y)|}{d^\alpha_P((t,x),(s,y))}\leq
 \frac{|u(t,x)-u(t,y)|}{d^\alpha_{c}(x,y)}+\frac{|u(t,y)-u(s,y)|}{|t-s|^\frac{\alpha}{2}}.$$

Since the case of the $t$-coordinate can be treated in a similar manner as the classical proof, combined with \eqref{927}, the \eqref{928} is valid.
\end{proof}
Recall the Poincar\'{e} inequality with respect to the spatial variable, stated below.
\begin{Prop}\label{poin}{\rm{(Theorem 2.1 in \cite{J})}}
There are positive constants $R_0$ and $C$ such that for any $0<R\leq R_0$, then
\begin{align}\label{p11}
\int_{B_{R}}|u-(u)_{B_R}|^2dx\leq CR^2\int_{B_R}|Xu|^2dx,
\end{align}
where $(u)_{B_R}$ is the average of $u$ on $B_R$.
\end{Prop}
Then we state an iterative lemma as follows.
\begin{Lem}{\rm{(Lemma 2.1 in \cite{G})\label{A4}}}
Let $\phi(t)$ be a nonnegative and nondecreasing function. Suppose that
$$\phi(\rho)\leq C_1\Big(\left(\frac{\rho}{R}\Big)^\alpha+\varepsilon\right)\phi(R)+C_2R^\beta,$$
for all $\rho< R\leq R_0$, with $C_1,\ C_2,\ \alpha,\ \beta$ nonnegative constants, $\beta<\alpha$. Then there exists constant $\varepsilon_0=\varepsilon_0(C_1,\ \alpha,\ \beta)$ such that if $\varepsilon<\varepsilon_0$, for all $\rho< R\leq R_0$, we have
$$\phi(\rho)\leq C\Big(\left(\frac{\rho}{R}\right)^\beta\phi(R)+C_2\rho^\beta\Big),$$
where $C$ is a positive constant depending on $C_1,\ \alpha,\ \beta$.
\end{Lem}

By the Rothschild-Stein lifting theorem given in \cite{RS77}, we state that a family of vector fields $\mathcal {X}=\{X_1,\ldots,X_q\}$ can be lifted by adding extra variables to free vector fields $Y=\{Y_1,\ldots,Y_q\}$. In order to simplify, we still denote the vector fields as $\mathcal {X}$. 
\begin{Lem}\label{xue2lem22}{\rm{(Lemma 2.2 in \cite{XZ97})}}
For any $x_0$ in $\Omega$ one can find coordinates in a neighborhood $V$ of $x_0$, and a matrix $T(x)\in GL(q,\mathbb{R})$ which is $C^\infty$ in $V$ such that if we set
$(Y_1,\ldots,Y_q)=T(t)(X_1,\ldots,X_q)$ we have

\noindent\rm{(i)} the set $(Y_1,\ldots,Y_q)$ satisfies the H\"{o}rmander's condition of order $r$ and is free up
to the order $r$ in $V$;

\noindent\rm{(ii)} $Y_j=\partial_{x_j}+\sum^{n}_{k=q+1}g_{jk}\partial_{x_k}$, $1\leq j\leq q$.
\end{Lem}
\begin{Lem}\label{xue2cor22}{\rm{(Lemma 2.3 in \cite{XZ97})}}
Let us set $(x',x'')\in\mathbb{R}^d$ with $x'
\in\mathbb{R}^q$ and for $k\in\mathbb{N}_0$ and $(C_J)_{|J|\leq k}$ in $\mathbb{R}^d$,
$$h(x')=\sum_{|J|\leq k}\frac{1}{J!}C_J(x'-x_0')^J.$$
Then for any $I$ with $|I|=k$, we have $Y_Ih=C_I$.
\end{Lem}

Then  we give a version of comparison principle for viscosity solutions as follows.
\begin{Prop}\label{CP}{\rm{(Theorem 4.4.5 in \cite{P09})}}
Let $v$ (resp. $w$) be a upper-semicontinuous viscosity subsolution (resp. lower-semicontinuous viscosity supersolution) with polynomial growth condition to the following HJE
\begin{equation}\label{hje0}
-\partial_t u-\frac{1}{2}{\rm{tr}}(\sigma\sigma^{T}(x) D^2u)+H(t,x,Du)=0,\quad\text{on}\ [0,T)\times\mathbb{R}^d,
\end{equation}
with a Hamiltonian
$H(t,x,Du)=\sup_{\alpha\in\mathcal {A}}\{b(x,\alpha)Du+f(t,x,\alpha)\}$ for $\mathcal {A}$ a subset of $\mathbb{R}^q$. Assume that the coefficient $b$ satisfies a linear growth condition in $x$ uniformly in $\alpha\in\mathcal {A}$.
 If $w(T,\cdot)\leq v(T,\cdot)$ on $\mathbb{R}^d$, then $w\leq v$ on $[0,T)\times\mathbb{R}^d$.
\end{Prop}

Finally, we recall the Schauder fixed point theorem and a variant of Lax-Milgram theorem as follows.
\begin{Lem}\label{APP133}{\rm{(Theorem 9.5 in \cite{An})}}
Let $E$ be a close bounded convex subset of a normed space $\mathcal {B}$. If $f:E\to\mathcal {B}$ is a compact map such that $f(E)\subseteq E$, then there exists a point $x\in E$  such that $f(x)=x$.
\end{Lem}
\begin{Lem}\label{App1.1}{\rm{(Corollary 5.8 in \cite{BH10})}}
Let $H$ be a Hilbert space and  $V$ be a dense subspace of $H$. Let $B(u,v)$ be a bilinear form on $H\times V$, if it satisfies

\noindent{\rm(i)}there exists a constant $M\geq0$ such that $|B(u,v)|\leq M\|u\|_H \|v\|_V$, for all $u\in H, v\in V$,

\noindent{\rm(ii)}there exists a constant $\delta>0$ such that $B(v,v)\geq \delta\|v\|_H^2$, for all $v\in V$,\\
then for any linear bounded functional $F(v)$ on $H$, there exists a unique $u\in H$, such that $F(v)=B(u,v)$. for all $ v\in V$.
\end{Lem}
\section{REGULARITY THEORY FOR DEGENERATE PARABOLIC EQUATIONS}
In this section, we shall show  a priori global Schauder estimates given in Proposition \ref{prop13}. Then we obtain the well-posedness and
regularity for linear degenerate parabolic equations \eqref{peq} given in Theorem \ref{eu}. Because the solution $u$ to the equation \eqref{peq2} is periodic, which is provided by Lemma \ref{periodic}, hence the proof of  Proposition \ref{prop13} is guaranteed  by the local estimates on $[0,T]\times\mathbb{R}^d$.

 To begin with, we first focus on the corresponding equation \eqref{peq}  defined on $[0,T]\times\mathbb{R}^d$ with all coefficients periodic as follows:
\begin{equation}\label{peq2}
\left\{
  \begin{aligned}
     & \mathcal {H}u(t,x)=f(t,x),&(t,x)\in (0,T]\times\mathbb{R}^d,  \\
   & u(0,x)=g(x)& x\in\mathbb{R}^d,
  \end{aligned}
\right.
\end{equation}
and show the estimates near the bottom boundary and the internal Schauder estimate. 

\subsection{ Proof of Proposition \ref{prop13}}
In the sequel, we shall use Einstein summation convention. Recall that $\Omega$ is a bounded subset of $\mathbb{R}^d$,
for any $t_0\in[0,T]$, $x_0\in\mathbb{R}^d$, $R>0$, set $$Q_{R}(t_0,x_0)=(t_0-R^2,t_0+R^2)\times B_{R}(x_0)\subset [0,T]\times\Omega,$$ $$Q^0_{R}(x_0)=(0,R^2)\times B_{R}(x_0)\subset[0,T]\times\Omega,$$
\begin{equation}\label{bound}
\partial_p Q_R(t_0,x_0)=\big\{\{t=t_0-R^2\}\times B_R(x_0)\big\}\cup\big\{(0,T)\times\partial B_R(x_0)\big\}.
\end{equation}
\begin{equation*}
\partial_p Q^0_R(x_0)=\big\{\{t=0\}\times B_R(x_0)\big\}\cup\big\{(0,T)\times\partial B_R(x_0)\big\}.
\end{equation*}
For convenience, we shall write $Q_R$ instead of $Q_{R}(t_0,x_0)$, $Q^0_R$ instead of $Q^0_{R}(x_0)$ and $B_R$ instead of $B_{R}(x_0)$. For the multi-index set $I$ and $i_k\in\{1,\ldots,q\}$, $k\in\mathbb{N}$, we denote
$$|X^ku|^2=\sum_{i_s\in I}|X_{i_1}X_{i_2}\ldots X_{i_k}u|^2.$$

To derive a priori estimates for linear equation \eqref{peq2},  we will simplify it from following three aspects. First, we can omit the lower terms of the equation \eqref{peq2}, which can be solved by the interpolation inequality.
Second, we assume that $g\equiv0$. Otherwise,  set $\tilde{u}=u-g$. Hence our discussion starts from the following linear equation
\begin{equation}\label{linear2}
\left\{
  \begin{aligned}
    &\partial_t u(t,x)-\Delta_{\mathcal {X}}u(t,x)=f(t,x),&(t,x)\in Q_R, \\
    &u(0,x)=0,& x\in B_R.
    \end{aligned}
\right.
\end{equation}
Third, we omit some proofs for the estimate near the bottom boundary that are similar to the interior estimate.

Now, we give  a parabolic version of the Poincar\'{e}-type inequality as follows.
\begin{Prop}\label{poin9}
For any $u\in C^{\infty}(Q_R)$ and $0<R\leq R_0$, $R_0$ is a positive constant, there exists a positive constant $C$ such that
\begin{align}\label{poincare91}
\iint_{Q_{R}}|u-{(u)}_{Q_R}|^2dxdt\leq CR^2\iint_{Q_R}|Xu|^2dxdt+CR^4\iint_{Q_{R}}|\partial_tu|^2dxdt,
\end{align}
where ${(u)}_{Q_{R}}=\frac{1}{|Q_{R}|}\iint_{Q_{R}}u(t,x)dxdt$ is the average of $u$ on $Q_{R}$.
\end{Prop}
\begin{proof}
Let ${(u)}_{B_R}(t)=\frac{1}{|B_R|}\int_{B_R}u(t,x)dx$, then we have
\begin{equation}\label{ur}
  {(u)}_{Q_{R}}=\frac{1}{|Q_{R}|}\iint_{Q_{R}} {(u)}_{B_R}(s)dxds.
\end{equation}
It follows from the Cauchy inequality and the Poincar\'{e} inequality \eqref{p11} that
\begin{align*}
\iint_{Q_{R}}|u- {(u)}_{Q_{R}}|^2dxdt&\leq 2\iint_{Q_{R}}|u- {(u)}_{B_R}(t)|^2dxdt
+2\iint_{Q_{R}}| {(u)}_{B_R}(t)- {(u)}_{Q_{R}}|^2dxdt\\
&\leq CR^2\iint_{Q_{R}}|Xu|^2dxdt+C\iint_{Q_{R}}| {(u)}_{B_R}(t)- {(u)}_{Q_{R}}|^2dxdt.
\end{align*}
For $t_0-R^2\leq s\leq t\leq t_0+R^2$, the H\"{o}lder inequality gives that
\begin{align*}
| {(u)}_{B_R}(t)- {(u)}_{B_R}(s)|^2&=\frac{1}{|B_R|^2}\Big(\int_{B_R}\int^t_s\partial_\tau u(\tau,x)d\tau dx\Big)^2\\
&\leq \frac{C|t-s|}{|B_R|}\int^t_s\int_{B_R}|\partial_\tau u(\tau,x)|^2dxd\tau\\
&\leq \frac{CR^2}{|B_R|}\iint_{Q_{R}}|\partial_\tau u(\tau,x)|^2dxd\tau.
\end{align*}
By using  \eqref{ur} and the H\"{o}lder inequality, then the estimate \eqref{poincare91} is valid.
\end{proof}

Next, we show the Caccioppoli-type inequality on intrinsic Sobolev spaces as follows.
\begin{Prop}\label{Ca31}
Let $u$ be a  solution of the equation \eqref{linear2}. Then for all $(t_0,x_0)\in [0,T]\times\mathbb{R}^d$, all $0<\rho\leq R<1$ such that $Q_{R}(t_0,x_0)\subset [0,T]\times\Omega$, we have
\begin{equation}\label{Ca20}
  \iint_{Q_{R/2}}|X^2u|^2dxdt\leq \frac{C}{R^4}\iint_{Q_{R}}u^2dxdt+C\iint_{Q_{R}}f^2dxdt.
\end{equation}
Furthermore, if $f\equiv0$, then for all $k\in\mathbb{N}$, we have $u\in C^\infty([0,T]\times\mathbb{R}^d)$ and
\begin{equation}\label{Caf0}
 \iint_{Q_{R/2^k}}|X^ku|^2dxdt\leq \frac{C}{R^{2k}}\iint_{Q_{R}}u^2dxdt.
\end{equation}
\end{Prop}
\begin{proof}
First, we prove \eqref{Ca20}. By the H\"{o}rmander theorem, the function $u$ is smooth, so $u\in H_{\mathcal {X}}^k(Q_R)$.
Let $\eta(x)\in C^\infty_0(B_R)$ be a cutoff function satisfying
  \begin{equation}\label{702}
  \eta\equiv1\ \text{on}\ B_{\rho},\ 0\leq\eta\leq1\ \text{on}\ B_R,\ \text{and}\ |X\eta|\leq \frac{C}{R-\rho}.
\end{equation}
Let $\xi(t)\in C^\infty(\mathbb{R})$ be a cutoff function satisfying
$$\xi\equiv1\ \ \text{on}\ t>t_0-\rho^2,\ \xi\equiv0\ \ \text{on}\ t<t_0-R^2,$$
\begin{equation}\label{cutofft}
   0\leq\xi\leq1\ \text{and}\  |\xi'(t)|\leq \frac{C}{(R-\rho)^2}\ \ \text{on}\ t_0-R^2\leq t\leq t_0-\rho^2.
\end{equation}
Multiply both sides of the equation \eqref{linear2} by $\xi^2\eta^2u$ and integrate on $Q_{R}$. \
Since $X^*_i=-X_i+c_i$, we get
\begin{align}\label{inte1}
\begin{split}
 &\iint_{Q_{R}}\xi^2\eta^2u\partial_t udxdt\\
=&\iint_{Q_{R}}X^*_i(\xi^2\eta^2u)X_i udxdt+\iint_{Q_{R}}\xi^2\eta^2ufdxdt\\\
=&\iint_{Q_{R}}\xi^2\big(c_i\eta^2u-\eta^2X_iu-2\eta X_i\eta u\big)X_iu dxdt+\iint_{Q_{R}}\xi^2\eta^2ufdxdt.
  \end{split}
\end{align}
For the left hand of \eqref{inte1}, using integration by parts, we have
 \begin{align}\label{left1}
 \begin{split}
 \iint_{Q_{R}}\xi^2\eta^2u\partial_t udxdt&=\frac{1}{2}\iint_{Q_{R}}\partial_t\big(\xi \eta u\big)^2 dxdt
 -\iint_{Q_{R}}\xi\xi'\eta^2u^2dxdt\\
 &=\frac{1}{2}\int_{B_{R}}\big(\xi \eta u\big)^2\big|_{t=t_0+R^2}dx
 -\iint_{Q_{R}}\xi\xi'\eta^2u^2dxdt\\
 &\geq  -\iint_{Q_{R}}\xi\xi'\eta^2u^2dxdt.
 \end{split}
 \end{align}
It follows from \eqref{inte1}, \eqref{left1} and $\varepsilon$-Cauchy inequality that
\begin{align*}
\begin{split}
 &\quad\iint_{Q_{R}}\xi^2\eta^2|Xu|^2dxdt\\
&\leq \iint_{Q_{R}}\xi\xi'\eta^2u^2dxdt+C\iint_{Q_{R}}\xi^2\eta^2|u||Xu|dxdt\\
&\quad+C\iint_{Q_{R}}\xi^2\eta|X\eta||u||Xu|dxdt+C\iint_{Q_{R}}\xi^2\eta^2|u||f|dxdt\\
&\leq \frac{C}{(R-\rho)^2}\iint_{Q_{R}}\xi\eta^2u^2dxdt+\frac{C}{R-\rho}\iint_{Q_{R}}\xi^2\eta|u||Xu|dxdt
+C\iint_{Q_{R}}\xi^2\eta^2|u||f|dxdt\\
&\leq  \frac{C}{(R-\rho)^2}\iint_{Q_{R}}\xi\eta^2u^2dxdt+ \frac{C_\varepsilon}{(R-\rho)^2}\iint_{Q_{R}}\xi^2\eta u^2dxdt+\varepsilon \iint_{Q_{R}}\xi^2\eta|Xu|^2dxdt\\
&\quad+ \frac{C}{(R-\rho)^2}\iint_{Q_{R}}\xi^2\eta^2u^2dxdt+(R-\rho)^2\iint_{Q_{R}}\xi^2\eta^2f^2dxdt.
\end{split}
\end{align*}
Taking $\varepsilon$ small enough, then we have
\begin{align}\label{Xuu}
\begin{split}
 \iint_{Q_{\rho}}|Xu|^2dxdt
 &\leq\iint_{Q_{R}}\xi^2\eta^2|Xu|^2dxdt\\
&\leq \frac{C}{(R-\rho)^2}\iint_{Q_{R}}u^2dxdt
+C(R-\rho)^2\iint_{Q_{R}}f^2dxdt.
\end{split}
\end{align}
Set the operator $$\mathcal {H}_0:=\partial_t-\Delta_\mathcal {X}.$$ For the equation \eqref{linear2}, since $\mathcal {H}_0u=f$, we have
$$\mathcal {H}_0(\xi\eta u)=\xi'\eta u+\xi\eta f-\sum^q_{i=1}\big(\xi u X_i^2\eta +2\xi X_i\eta X_iu\big).$$
Then we have
\begin{align*}
\begin{split}
&\quad \iint_{Q_{R}}|\mathcal {H}_0(\xi\eta u)|^2dxdt\\
&\leq \frac{C}{(R-\rho)^4}\iint_{Q_{R}}u^2dxdt+\frac{C}{(R-\rho)^2}\iint_{Q_{R}}|Xu|^2dxdt
+C\iint_{Q_{R}}f^2dxdt.
\end{split}
\end{align*}
Using the maximum regularity for $\mathcal {H}_0$ proved by Theorem 18 of \cite{RS77}, we get
\begin{align*}
\|X^2u\|_{L^2(Q_{\rho})}&\leq\|\xi\eta u\|_{H^{2}_{\mathcal {X}}(Q_{R})}\\
&\leq C\|\mathcal {H}_0(\xi\eta u)\|_{L^2(Q_{{R}})}+C\|\xi\eta u\|_{L^2(Q_R )}\\
&\leq \frac{C}{(R-\rho)^4}\iint_{Q_{R}}u^2dxdt+\frac{C}{(R-\rho)^2}\iint_{Q_{R}}|Xu|^2dxdt+C\iint_{Q_{R}}f^2dxdt.
\end{align*}
Taking $\rho=\frac{R}{2},\ R=\frac{3R}{4}$ in the above estimate, taking $\rho=\frac{3R}{4}$ in the estimate \eqref{Xuu}, then we have
\begin{align*}
&\quad\iint_{Q_{R/2}}|X^2u|^2dxdt\\
&\leq\frac{C}{R^4}\iint_{Q_{3R/4}}u^2dxdt+\frac{C}{R^2}\iint_{Q_{3R/4}}|Xu|^2dxdt+C\iint_{Q_{3R/4}}f^2dxdt\\
&\leq\frac{C}{R^4}\iint_{Q_{R}}u^2dxdt+C\iint_{Q_{R}}f^2dxdt,
\end{align*}
which implies that the estimate \eqref{Ca20}  is valid.

In the following, let $\rho=\frac{R}{2}$ and $f\equiv0$, we prove \eqref{Caf0} by induction on $k$, which is inspired of Proposition 3.1 in \cite{{XZ97}}.
First, we note that the estimate \eqref{Xuu} corresponds to
\begin{equation}\label{xue29822}
  \iint_{Q_{R/2}}|Xu|^2dxdt\leq\frac{C}{R^2}\iint_{Q_{R/2}}u^2dxdt,
\end{equation}
which implies that \eqref{Ca20} is valid for $k=1$.

Assume now \eqref{Caf0} true up to the order $k-1$, that is for any $ i\in\{2,\ldots,k-1\}$,
\begin{equation}\label{Xiuu}
\iint_{Q_{R/2^i}}|X^iu|^2dxdt\leq \frac{C}{R^{2i}}\iint_{Q_{R}}u^2dxdt.
\end{equation}
Let $\eta(x)\in C^{\infty}_{0}(B_{{R}/{2^{k-1}}})$ be a cutoff function satisfying
\begin{equation*}
 \eta\equiv1\ \text{on}\ B_{{{R}}/{2^k}},\ 0\leq\eta\leq1\ \text{on}\ B_{R}\ \text{and}\ |X^{I}\eta|\leq CR^{-|I|}.
\end{equation*}
Let $\xi(t)\in C^\infty(\mathbb{R})$ be a cutoff function satisfying
$$\xi\equiv1\ \ \text{on}\ t>t_0-\frac{R^2}{4^k},\ \xi\equiv0\ \ \text{on}\ t<t_0-R^2,$$
\begin{equation*}
   0\leq\xi\leq1\ \text{and}\  |\partial_t^i\xi|\leq CR^{-2i}\ \text{on}\ t_0-R^2\leq t\leq t_0-\frac{R^2}{4^k}.
\end{equation*}
Since $\mathcal {H}_0u=0$, then $\mathcal {H}_0(\xi\eta u)=\xi'\eta u-\xi\Delta_\mathcal {X}\eta u-2\xi X\eta Xu$ and \eqref{Xiuu} gives that
\begin{align*}
  \|\mathcal {H}_0(\xi\eta u)\|_{H^{k-2}_{\mathcal {X}}(Q_{{R}/{2^{k-1}}})}
  &\leq C\sum_{|I|\leq k-1}R^{-k+|I|}\|X^Iu\|_{L^2(Q_{{R}/{2^{k-1}}})}\leq \frac{C}{R^{k}}\iint_{Q_{R}}u^2dxdt.
\end{align*}
Using the maximum regularity for $\mathcal {H}_0$ (Theorem 18 in \cite{RS77}) again, we have
\begin{align*}
\sum_{|I|=k}\|X^Iu\|_{L^2(Q_{{R}/{2^{k}}})}
&\leq C\|\mathcal {H}_0(\xi\eta u)\|_{H^{k-2}_{\mathcal {X}}(Q_{{R}/{2^{k-1}}})}+C\|\xi\eta u\|_{L^2(Q_{{R}/{2^{k-1}}})}\\
&\leq \frac{C}{R^{k}}\iint_{Q_{R}}u^2dxdt.
\end{align*}
Hence the estimate \eqref{Caf0} is valid for $k$, and the result follows.
\end{proof}
The corresponding estimate  near the bottom boundary is obtained as follows.
\begin{Prop}\label{0Ca20}
Let $u$ be a  solution of the equation \eqref{linear2}. Then for all $(t_0,x_0)\in [0,T]\times\mathbb{R}^d$, all $0<R<1$ such that $Q^0_R(0,x_0)\subset [0,T]\times\Omega$, we have
\begin{equation}
  \iint_{Q^0_{R/2}}|X^2u|^2dxdt\leq \frac{C}{R^4}\iint_{Q^0_R}u^2dxdt+C\iint_{Q^0_R}f^2dxdt.
\end{equation}
Furthermore, if $f\equiv0$, then for all $k\in\mathbb{N}$, we have $u\in C^\infty([0,T]\times\mathbb{R}^d)$ and
\begin{equation}\label{0Caf0}
 \iint_{Q^0_{R/2^k}}|X^ku|^2dxdt\leq \frac{C}{R^{2k}}\iint_{Q^0_{R}}u^2dxdt.
\end{equation}
\end{Prop}
\begin{proof}
The proof is similar to Proposition \ref{Ca31}. The difference is that, it is only to multiply the equation \eqref{linear2} by $\eta^2u$  and integrate on $Q^0_R$, since $u(0,x)=0$. Here we omit it.
\end{proof}
\begin{Prop}\label{102}
Let $u$ be a  solution of the equation \eqref{linear2} with $f\equiv0$. Then for all $(t_0,x_0)\in [0,T]\times\mathbb{R}^d$, all $0<R<1$ such that $Q_{R}(t_0,x_0)\subset [0,T]\times\Omega$, we have
\begin{align}\label{108}
\iint_{Q_{R/2}}|\partial_tu|^2dxdt&\leq\frac{C}{R^2}\iint_{Q_{R}}|Xu|^2dxdt.
\end{align}
\end{Prop}
\begin{proof}
Let $\xi$ and $\eta$ be cutoff functions choosen in the proof of Proposition \ref{Ca31}. Multiply the equation \eqref{linear2} with $f\equiv0$ by $\xi^2\eta^2\partial_t u$ and integrate on $Q_{R}$. By $X^*_i=-X_i+c_i$, we get
\begin{align*}
\iint_{Q_{R}}\xi^2\eta^2|\partial_t u|^2dxdt
=&\iint_{Q_{R}}X_i^*(\xi^2\eta^2\partial_tu)X_iudxdt\\
=&-\iint_{Q_{R}}\xi^2\eta^2X_iuX_i\partial_tudxdt
-2\iint_{Q_{R}}\xi^2\eta X_i\eta\partial_tuX_iu dxdt\\
&+\iint_{Q_{R}}c_i\xi^2\eta^2\partial_t uX_iudxdt.
\end{align*}
Similar to \eqref{left1}, we have
\begin{align*}
-\iint_{Q_{R}}\xi^2\eta^2X_iuX_i\partial_tudxdt
&=-\frac{1}{2}\int_{B_{R}}\xi^2\eta^2|X_iu|^2\Big|_{t=t_0+R^2}dx
+\iint_{Q_{R}}\xi\xi'\eta^2|X_iu|^2dxdt\\
&\leq\iint_{Q_{R}}\xi\xi'\eta^2|Xu|^2dxdt.
\end{align*}
Then the $\varepsilon$-Cauchy inequality gives that
\begin{align*}
&\quad\iint_{Q_{R}}\xi^2\eta^2|\partial_t u|^2dxdt\\
&\leq C\iint_{Q_{R}}\xi|\xi'|\eta^2|Xu|^2dxdt+C\iint_{Q_{R}}\xi^2\eta \big(\eta+|X\eta|\big)|\partial_tu||Xu| dxdt\\
&\leq C\iint_{Q_{R}}\xi|\xi'|\eta^2|Xu|^2dxdt+2\varepsilon\iint_{Q_{R}}\xi^2\eta^2|\partial_t u|^2dxdt\\
&\quad +C_\varepsilon\iint_{Q_{R}}\xi^2\big(\eta^2+|X\eta|^2\big)|Xu|^2dxdt\\
&\leq 2\varepsilon\iint_{Q_{R}}\xi^2\eta^2|\partial_t u|^2dxdt+\frac{C}{R^2}\iint_{Q_{R}}|Xu|^2dxdt.
\end{align*}
Taking $\varepsilon$ small enough, we have
\begin{align*}
\iint_{Q_{R/2}}|\partial_tu|^2dxdt&\leq\iint_{Q_{R}}\xi^2\eta^2|\partial_tu|^2dxdt\leq \frac{C}{R^{2}}\iint_{Q_{R}}|Xu|^2dxdt,
 \end{align*}
and the result follows.
\end{proof}
\begin{Prop}\label{X2uuR}
Let $u$ be a  solution of the equation \eqref{linear2} with $f\equiv0$. Then for all $(t_0,x_0)\in[0,T]\times\mathbb{R}^d)$, all  $0<\rho\leq R<1$ such that $Q_{R}(t_0,x_0)\subset [0,T]\times\Omega$, there exists a positive constant $C$ such that
\begin{equation}\label{X2uR}
\iint_{Q_\rho}|X^2 u-(X^2 u)_{Q_{R}}|^2dxdt\leq C\left(\frac{\rho}{R}\right)^{Q+4}\iint_{Q_R}|X^2 u-(X^2 u)_{Q_{R}}|^2dxdt,
\end{equation}
where $Q$ is the homogeneous dimension with respect to the vector fields $\mathcal{X}$.
\end{Prop}
\begin{proof}
To begin with, we recall that  Proposition 2.4 in \cite{XZ97} shows that for any $t\in[0,T]$, $k>\frac{Q}{2}$, we have
\begin{equation}\label{equation1}
\sup\limits_{x\in  B_{R/4}}|u(t,x)|\leq C\sum\limits_{|I|\leq k}R^{|I|-\frac{Q}{2}}\|X^Iu(t,x)\|_{L^2(B_R)},
\end{equation}
and for any $x\in\mathbb{R}^d$, $k>1$, we have
\begin{equation}\label{equation2}
\sup\limits_{t\in (0,\frac{R^2}{16})}|u(t,x)|\leq C\sum\limits_{2i\leq k}R^{2i-1}\|\partial^i_tu(t,x)\|_{L^2(0,R^2)}.
\end{equation}

Let $k,k_1,k_2$ be some nonnegative integers.  For the case $0<\rho\leq\frac{R}{2^{k_1+k_2+k+2}}$, it follows from \eqref{equation1}, \eqref{equation2} and \eqref{Caf0} that

\begin{align}\label{354}
\begin{split}
&\quad\iint_{Q_\rho}|X^k u|^2dxdt\\
   &\leq \int^{\rho^2}_{0}|B_\rho|\sup_{x\in B_{R/2^{k_1+k_2+k+2}}}|X^k u|^2dt\\
  &\leq C\int^{\rho^2}_{0}|B_\rho|\sum_{|I|\leq k_1}R^{2|I|-Q}\int_{B_{R/2^{k_1+k_2+k}}}|X^IX^k u|^2dxdt\\
  &\leq C\rho^Q\sum_{|I|\leq k_1}R^{2|I|-Q}\int_{B_{R/2^{k_1+k_2+k}}}\rho^2\sup_{t\in(0,\rho^2)}|X^{I}X^k u|^2dx\\
  &\leq  C\rho^2\left(\frac{\rho}{R}\right)^{Q}\sum_{|I|\leq k_1}R^{2|I|}\int_{B_{R/2^{k_1+k_2+k}}}\sum_{2i\leq k_2}R^{4i-2}\int^{\frac{R^2}{4^{k_1+k_2+k}}}_{0}|\partial^i_tX^{I}X^k u|^2dt dx\\
    &\leq  C\left(\frac{\rho}{R}\right)^{Q+2}\sum_{|I|+2i\leq k_1+k_2}R^{2|I|+4i}\iint_{Q_{R/2^{k_1+k_2+k}}}|\partial^i_tX^{I}X^k u|^2dxdt\\\
  &\leq \frac{C}{R^{2k}}\left(\frac{\rho}{R}\right)^{Q+2}\iint_{Q_{R}}u^2dxdt.
  \end{split}
 \end{align}
Since $u-(u)_{Q_{R/2}}$ is still a solution to the equation \eqref{linear2}, for any $k\geq1$, we have
 $$\iint_{Q_\rho}|X^k u|^2dxdt \leq \frac{C}{R^{2k}}\left(\frac{\rho}{R}\right)^{Q+2}\iint_{Q_{R/2}}|u-(u)_{Q_{R/2}}|^2dxdt.$$
 It follows from the Poincar\'{e}-type inequality given in Proposition \ref{poin9} and the estimate \eqref{108} that
 \begin{align*}
\iint_{Q_\rho}|X^ku|^2dxdt&\leq \frac{C}{R^{2k}}\left(\frac{\rho}{R}\right)^{Q+2}\Big(R^2\iint_{Q_{R/2}}|Xu|^2dxdt
+R^4\iint_{Q_{R/2}}|\partial_t u|^2dxdt\Big)\\
&\leq \frac{C}{R^{2k-2}}\left(\frac{\rho}{R}\right)^{Q+2}\iint_{Q_{R}}|Xu|^2dxdt.
\end{align*}
By using Lemma \ref{xue2cor22}, take $$h(x)=\sum^{q}_{i=1} {(X_iu)}_{Q_{R}}(x'_{i}-x'_{0i}),$$
then we have $\partial_t h-\Delta_\mathcal {X}h=0$ and $X_ih= {(X_iu)}_{Q_{R}}$, which implies that $u-h$ is still a solution of the equation \eqref{linear2}. Then  for $k\geq2$, we have
\begin{align}\label{109}
\begin{split}
&\quad\iint_{Q_{\rho}}|X^ku|^2dxdt\\&\leq \frac{C}{R^{2k-2}}\left(\frac{\rho}{R}\right)^{Q+2}\iint_{Q_{R}}|Xu-(Xu)_{Q_{R}}|^2dxdt\\
&\leq \frac{C}{R^{2k-2}}\left(\frac{\rho}{R}\right)^{Q+2}\Big(R^2\iint_{Q_{R}}|X^2u|^2dxdt
+R^4\iint_{Q_{R}}|\partial_tXu|^2dxdt\Big),
\end{split}
\end{align}
the last inequality is provided by the Poincar\'{e} type inequality \eqref{poincare91}.

Since $\partial_tu$ is a solution of the equation \eqref{linear2} on $Q_{R}$, then the inequality \eqref{xue29822} is also valid for $\partial_tu$, that is
\begin{align}\label{222}
\iint_{Q_{R/2}}|X\partial_tu|^2dxdt\leq \frac{C}{R^2}\iint_{Q_{R}}|\partial_t u|^2dxdt\leq \frac{C}{R^2}\iint_{Q_{R}}|X^2u|^2dxdt.
\end{align}
Combined  \eqref{109} with \eqref{222}, we get
\begin{align*}
\iint_{Q_{\rho}}|X^ku|^2dxdt\leq \frac{C}{R^{2k-4}}\left(\frac{\rho}{R}\right)^{Q+2}\iint_{Q_{R}}|X^2u|^2dxdt.
\end{align*}
Up to now, the derivative order on the right hand of the last estimate  has increased to second order. Then we will construct the solution and obtain an estimation form containing the mean value. By using Lemma \ref{xue2cor22} again, take
 $$h(x)= t\sum^{q}_{i=1}{(X^2_iu)}_{Q_{R}} -\frac{1}{2}\sum^{q}_{i,j=1} {(X_iX_ju)}_{Q_{R}}(x'_{i}-x'_{0i})(x'_{j}-x'_{0j}).$$
 Since $\partial_t h-\Delta_\mathcal {X}h=0$ and $X_iX_jh= {(X_iX_ju)}_{Q_{R}}$, then $u-h$ still a solution of the equation \eqref{linear2}. If $k\geq3$, then we have
\begin{align}\label{XkX2u}
\iint_{Q_{\rho}}|X^ku|^2dxdt\leq \frac{C}{R^{2k-4}}\left(\frac{\rho}{R}\right)^{Q+2}\iint_{Q_{R}}|X^2u- {(X^2u)}_{Q_{R}}|^2dxdt.
\end{align}
Furthermore, combined the estimate \eqref{XkX2u} with the Poincar\'{e} type inequality \eqref{poincare91} again, since $\partial_t u=\Delta_\mathcal {X}u$, we have
\begin{align*}
&\iint_{Q_{\rho}}|X^2u- {(X^2u)}_{Q_\rho}|^2dxdt\\
\leq& C\Big(\rho^2\iint_{Q_{R}}|X^3u|^2dxdt+\rho^4\iint_{Q_{R}}|\partial_tX^2u|^2dxdt\Big)\\
\leq& C\Big(\rho^2\iint_{Q_{R}}|X^3u|^2dxdt+\rho^4\iint_{Q_{R}}|X^4u|^2dxdt\Big)\\
\leq& C\left(\frac{\rho}{R}\right)^{Q+2}\left[\left(\frac{\rho}{R}\right)^2
+\left(\frac{\rho}{R}\right)^4\right]\iint_{Q_{R}}|X^2u- {(X^2u)}_{Q_{R}}|^2dxdt,
\end{align*}
which implies that the estimate \eqref{X2uR} is valid for any $0<\rho\leq\frac{R}{2^{k_1+k_2+k+2}}$. For the case $\frac{R}{2^{k_1+k_2+k+2}}<\rho\leq R$, just take $C=2^{(k_1+k_2+k+2)(Q+4)}$, then the estimate \eqref{84} holds and the result follows.
\end{proof}
The corresponding result on the ball $Q^0_R$ holds as follows.
\begin{Prop}\label{X2uuR0}
Let $u$ be a solution of the equation \eqref{linear2} with $f\equiv0$. Then for all $(t_0,x_0)\in[0,T]\times\mathbb{R}^d$, all  $0<\rho\leq R<1$, $Q_{R}(0,x_0)\subset [0,T]\times\Omega$, there exists a positive constant $C$ such that
\begin{equation}\label{84}
\iint_{Q^0_\rho}|X^2 u|^2dxdt\leq C\left(\frac{\rho}{R}\right)^{Q+4}\iint_{Q^0_R}|X^2 u|^2dxdt.
\end{equation}
\end{Prop}
\begin{proof}
If $\rho$ is big enough, it is easy to check that the inequality \eqref{84} holds. Otherwise, by the same argument as the estimate \eqref{354}, we have
\begin{equation}\label{727}
   \iint_{Q^0_\rho}|X^k u|^2dxdt \leq \frac{C}{R^{2k}}\left(\frac{\rho}{R}\right)^{Q+2}\iint_{Q^0_R}u^2dxdt.
\end{equation}
 Now we claim that
\begin{equation}\label{t}
 \iint_{Q^0_\rho}|X^k u|^2dxdt\leq C\rho^4\iint_{Q^0_\rho}|\partial_tX^k u|^2 dxdt.
\end{equation}
In fact, since $u\big|_{t=0}=0$, we have $X^k u\big|_{t=0}=0$. Then by using the Leibniz formula and the H\"{o}lder inequality, we have
 \begin{align*}
   \int^{\rho^2}_0|X^k u(s)|^2 ds&= \int^{\rho^2}_0\Big(\int^s_0\partial_tX^k u(\xi)d\xi\Big)^2 ds\\
   &\leq \int^{\rho^2}_0s\int^s_0|\partial_tX^k u(\xi)|^2d\xi ds\\
   &\leq C\rho^4\int^{\rho^2}_0|\partial_tX^k u(\xi)|^2d\xi,
 \end{align*}
which implies that the claim \eqref{t} is valid.

Since $|\partial_tX^2 u|\leq C|X^4u|$, it follows from \eqref{727} and \eqref{t} that
\begin{align*}
  \iint_{Q^0_{\rho}}|X^2 u|^2dxdt &\leq C\rho^4\iint_{Q^0_\rho}|X^4u|^2dxdt\\
  &\leq C\rho^4\left(\frac{\rho}{R}\right)^{Q+2}\frac{1}{R^{8}}\iint_{Q^0_{R}}u^2dxdt\\
   &\leq C\left(\frac{\rho}{R}\right)^{Q+6}\iint_{Q^0_{R}}|X^2 u|^2dxdt,
\end{align*}
the last inequality is provided by \eqref{t} with $k=0$. Hence the estimate \eqref{84} is valid.
\end{proof}

Next, we make the decomposition of the solution $u$ to the equation \eqref{linear2}  on $Q_{R}$ as follows
$$u=u_1+u_2,$$
where $u_1$ satisfies
\begin{equation}\label{355}
\left\{
 \begin{array}{ll}
  \partial_t u_1-\Delta_\mathcal {X}u_1=(f)_{Q_{R}}, \ (t,x)\in Q_{R},\\
  u_1\big|_{\partial_p Q_{R}}=u\big|_{\partial_p Q_{R}},
 \end{array}
\right.
\end{equation}
 and $u_2$ satisfies
\begin{equation}\label{356}
\left\{
 \begin{array}{ll}
  \partial_t u_{2}-\Delta_\mathcal {X}u_{2}=f-(f)_{Q_{R}}, \ (t,x)\in Q_{R},\\
u_{2}\big|_{\partial_p Q_{R}}=0.
 \end{array}
\right.
\end{equation}
Here $(f)_{Q_{R}}$ is the average of $f$ on $Q_{R}$ and $\partial_p Q_{R}$ is defined in \eqref{bound}, for any $0<R<1$.

Then we give some estimates for $u_{2}$ as follows.
 \begin{Prop}\label{pro1}
Let  $u_2$ be a solution of the equation \eqref{356} on $Q_{R}$. For $0<R\leq1$, there exists a positive constant $C$ such that
\begin{align}\label{Xu2UR}
\begin{split}
\iint_{Q_{R}}|Xu_{2}|^2dxdt&\leq CR^2\iint_{Q_{R}}|f-(f)_{Q_{R}}|^2dxdt.
\end{split}
\end{align}
\end{Prop}
\begin{proof}
Multiply both sides of the equation \eqref{356} by $u_{2}$ and integrate on $Q_{R}$, then we get
\begin{align}\label{953}
\begin{split}
     \iint_{Q_{R}}u_{2}\partial_tu_{2}dxdt-\iint_{Q_{R}}X^*_iu_{2}X_iu_{2}dxdt
     =\iint_{Q_{R}}(f-(f)_{Q_{R}})u_{2}dxdt.
\end{split}
\end{align}
For the left hand of \eqref{953}, since $u_{2}(0,x)=0$ for any $x\in B_R$, we have
\begin{align*}
    \iint_{Q_{R}}u_{2}\partial_tu_{2}dxdt
=\frac{1}{2}\iint_{Q_{R}}\partial_tu_{2}^2dxdt=\frac{1}{2}\int_{B_{R}}u_{2}^2\Big|_{t=R^2}dx\geq0.
\end{align*}
Sine $X_i^*=-X_i+c_i$, then it follows from  \eqref{953} that
\begin{align}\label{954}
\begin{split}
    \iint_{Q_{R}}|X_iu_{2}|^2dxdt\leq \iint_{Q_{R}}|f-(f)_{Q_{R}}||u_{2}|dxdt+\iint_{Q_{R}}|c_i||u_{2}||X_iu_{2}|dxdt.
\end{split}
\end{align}
Since $u_2=0$ on $\partial_pQ_{R}$,
using the Poincar\'{e} inequality \eqref{p11} on the spatial variable, we have
\begin{align}\label{p1}
\begin{split}
\iint_{Q_{R}}u_{2}^2dxdt
\leq  CR^2\iint_{Q_{R}}|Xu_2|^2dxdt.
\end{split}
\end{align}
It follows from the H\"{o}lder inequality and \eqref{p1} that
 \begin{align}\label{956}
\begin{split}
  \iint_{Q_{R}}|c_i||u_{2}||X_iu_{2}|dxdt
  &\leq C\Big(\iint_{Q_{R}}u_{2}^2dxdt\Big)^{\frac{1}{2}}\Big(\iint_{Q_{R}}|Xu_{2}|^2dxdt\Big)^{\frac{1}{2}}\\
  &\leq C R\iint_{Q_{R}}|Xu_{2}|^2dxdt.
  \end{split}
\end{align}
Similarly, combining the $\varepsilon$-Cauchy inequality with \eqref{p1}, we have
\begin{align}\label{29}
\begin{split}
   \iint_{Q_{R}}|f-(f)_{Q_{R}}||u_{2}|dxdt&\leq \frac{\varepsilon}{R^2}\iint_{Q_{R}}u_{2}^2dxdt+CR^2\iint_{Q_{R}}|f-(f)_{Q_{R}}|^2dxdt\\
  &\leq \varepsilon\iint_{Q_{R}}|Xu_{2}|^2dxdt+C R^2 \iint_{Q_{R}}|f-(f)_{Q_{R}}|^2dxdt.
  \end{split}
\end{align}
By choosing  $R$ such that $C R\leq\varepsilon$, then using \eqref{29} and \eqref{956} to \eqref{954}, we have
\begin{align*}
\iint_{Q_{R}}|Xu_{2}|^2dxdt
 &\leq 2\varepsilon\iint_{Q_{R}}|Xu_{2}|^2dxdt+CR^2\iint_{Q_{R}}|f-(f)_{Q_{R}}|^2dxdt.
\end{align*}
Taking $\varepsilon$ small enough, the estimate \eqref{Xu2UR} is valid.
\end{proof}

 By the same argument, the inequality \eqref{Xu2UR} is also valid on $Q^0_R$ as follows.
  \begin{Prop}\label{pro2}
Let  $u_2$ be a solution of the equation \eqref{356} on $Q_{R}^0$. For $0<R\leq1$, there exists a positive constant $C$ such that
\begin{align*}
\iint_{Q^0_{R}}|Xu_{2}|^2dxdt&\leq CR^2\iint_{Q^0_{R}}|f-(f)_{Q_{R}}|^2dxdt.
\end{align*}
\end{Prop}

Then we give some useful estimates for $u$ as follows.
\begin{Prop}\label{729}
Let $u$ be a  solution of the equation \eqref{linear2} on $Q_{R}$. Then for all $(t_0,x_0)\in[0,T]\times\mathbb{R}^d$, all  $0<\rho\leq R<1$ such that $Q_{R}(t_0,x_0)\subset [0,T]\times\Omega$, there exists a positive constant $C$ such that
\begin{align}\label{1105}
\begin{split}
&\quad\iint_{Q_\rho}|X^2u-(X^2u )_{Q_{\rho}}|^2dxdt\\
&\leq C\left(\frac{\rho}{R}\right)^{Q+2+2\alpha}\iint_{Q_{R/2}}|X^2 u-(X^2u)_{Q_{R/2}}|^2dxdt
+\rho^{Q+2+2\alpha}[f]^2_{C^{0,\alpha}_{\mathcal {X}} (Q_{3R/4})}.
\end{split}
\end{align}
\end{Prop}
\begin{proof}
 Recall the decomposition  $u=u_1+u_{2}$, where $u_1$ and $u_2$ are solutions of the equations \eqref{355} and \eqref{356} on $Q_{R}$, respectively. For the case $0<\rho\leq\frac{R}{2}$, set $$v=u_1-t(f)_{Q_{R}},$$ then $v$ satisfies the linear equation \eqref{linear2} with $f\equiv0$, that is,
\begin{equation}\label{ee}
\left\{
  \begin{aligned}
    &\partial_t v-\Delta_{\mathcal {X}}v=0,& (t,x)\in Q_{R}, \\
    &  v(0,x)=0,& x\in B_R.
    \end{aligned}
\right.
\end{equation}
It is easy to show  that for any $\ k\in\mathbb{N}_+$,
$X^ku_1=X^kv$.

For the estimate for $u_1$, Proposition \ref{X2uuR} gives that
\begin{align}\label{81}
\begin{split}
\iint_{Q_\rho}|X^2u_ 1- {(X^2u_ 1)}_{Q_{R}}|^2dxdt&=\iint_{Q_\rho}|X^2v- {(X^2v)}_{Q_{R}}|^2dxdt\\
&\leq C\left(\frac{\rho}{R}\right)^{Q+4}\iint_{Q_{R}}|X^2v- {(X^2v)}_{Q_{R}}|^2dxdt\\
&= C\left(\frac{\rho}{R}\right)^{Q+4}\iint_{Q_{R}}|X^2u_1- {(X^2u_ 1)}_{Q_{R}}|^2dxdt.
\end{split}
\end{align}

For the estimate for $u_2$, by replacing $f$ with $f-(f)_{Q_{R}}$ in the inequality \eqref{Ca20}, then for any $0<\rho<\frac{R}{2}$, we have
\begin{align}\label{371}
\begin{split}
\iint_{Q_\rho}|X^2u_{2}-(X^2u_{2})_{Q_\rho}|^2dxdt&\leq \iint_{Q_{R/2}}|X^2u_{2}|^2dxdt\\
&\leq \frac{C}{R^4}\iint_{Q_{R}}u_{2}^2dxdt+C\iint_{Q_{R}}|f-(f)_{Q_{R}}|^2dxdt.
\end{split}
\end{align}
Since  $u_{2}\big|_{\partial_p Q_{R}}=0$, using the Poincar\'{e} inequality \eqref{p11}  for $u_2$, we have
\begin{equation}\label{2121}
    \iint_{Q_{R}}u_{2}^2dx\leq CR^2\iint_{Q_R}|Xu_2|^2dx.
\end{equation}
It follows from inequalities \eqref{371}, \eqref{2121} and \eqref{Xu2UR} that
\begin{align}\label{xue283}
\begin{split}
\iint_{Q_\rho}|X^2u_{2}-(X^2u_{2})_{Q_\rho}|^2dxdt
&\leq \frac{C}{R^2}\iint_{Q_{R}}|Xu_{2}|^2dxdt+C\iint_{Q_{R}}|f-(f)_{Q_{R}}|^2dxdt\\
&\leq  C\iint_{Q_{R}}|f-(f)_{Q_{R}}|^2dxdt\\
&\leq CR^{Q+2+2\alpha}[f]^2_{C^{0,\alpha}_{\mathcal {X}} (Q_{R})},
\end{split}
\end{align}
the last inequality is provided by Lemma \ref{A3}.

\begin{align*}
&\quad\iint_{Q_\rho}|X^2u-(X^2u)_{Q_\rho}|^2dxdt\\
&\leq 2\iint_{Q_\rho}|X^2u_1-(X^2u_{1})_{Q_\rho}|^2dxdt+2\iint_{Q_\rho}|X^2u_{2}-(X^2u_{2})_{Q_\rho}|^2dxdt\\
&\leq C\left(\frac{\rho}{R}\right)^{Q+4}\iint_{Q_{R/2}}|X^2u_1-(X^2u_{1})_{Q_{R/2}}|^2dxdt
+C\iint_{Q_{R/2}}|X^2u_{2}-(X^2u_{2})_{Q_{R/2}}|^2dxdt\\
&\leq C\left(\frac{\rho}{R}\right)^{Q+4}\iint_{Q_{R/2}}|X^2u-(X^2u)_{Q_{R/2}}|^2dxdt
+C\iint_{Q_{R/2}}|X^2u_{2}-(X^2u_{2})_{Q_{R/2}}|^2dxdt\\
  &\leq C\left(\frac{\rho}{R}\right)^{Q+4}\iint_{Q_{R/2}}|X^2u-(X^2u)_{Q_{R/2}}|^2dxdt+ CR^{Q+2+2\alpha}[f]^2_{C^{0,\alpha}_{\mathcal {X}} (Q_{3R/4})}.
\end{align*}
Hence by using  Lemma \ref{A4}, the estimate \eqref{1105} holds.
For the case $\frac{R}{2}<\rho\leq R$, just taking $C=2^{Q+2+2\alpha}$, then the estimate \eqref{1105} is valid.
\end{proof}
\begin{Prop}\label{7210}
Let $u$ be a  solution of the equation \eqref{linear2} $Q_{R}$. Then for all $(t_0,x_0)\in[0,T]\times\mathbb{R}^d$, all  $0<\rho\leq \frac{R}{2}<R<1$ such that $Q_{R}(t_0,x_0)\subset [0,T]\times\Omega$, there exists a positive constant $C$ such that
\begin{equation}\label{bound1}
  [X^2 u]_{C^{0,\alpha}_{\mathcal {X}}(Q_{\rho})}
   \leq C\Big(\frac{1}{R^{2+\alpha}}\|u\|_{L^{\infty}({Q_R})}
  +\frac{1}{R^{\alpha}}\|f\|_{L^{\infty}({Q_R})}
  +[f]_{C^{0,\alpha}_{\mathcal {X}} ({Q_R})}\Big).
\end{equation}
\end{Prop}
\begin{proof}
It follows from \eqref{1105} and  Proposition \ref{Ca31} that for any $0<\rho<\frac{R}{2}$,
\begin{align*}
&\quad\iint_{Q_{\rho}}\big|X^2u-(X^2u)_{Q_\rho}\big|^2dxdt\\
  &\leq\rho^{Q+2+2\alpha}\Big( \frac{1}{R^{Q+2+2\alpha}}\iint_{Q_{R/2}}|X^2 u|^2dxdt+[f]^2_{C^{0,\alpha}_{\mathcal {X}} (Q_{3R/4})}\Big)\\
  &\leq C\rho^{Q+2+2\alpha}\Big(\frac{1}{R^{Q+6+2\alpha}}\iint_{Q_{R}}u^2dxdt
  +\frac{1}{R^{2\alpha}}\|f\|^2_{L^\infty(Q_{R})}+[f]^2_{C^{0,\alpha}_{\mathcal {X}} (Q_{R})}\Big).
\end{align*}
On the other hand, for any $(t,x)\in Q_R$, we have
\begin{align*}
   &\quad \iint_{Q_{R}(x)\cap Q_{R}}|X^2u(s,y)- {(X^2u)}_{Q_{R}(t,x) }|^2dyds\\
   &\leq\iint_{Q_{R}}|X^2u(s,y)- {(X^2u)}_{Q_{R}}|^2dyds\\
   &\leq C\rho^{Q+2+2\alpha}\Big(\frac{1}{R^{4+2\alpha}}\|u\|^2_{L^{\infty}(Q_{R}(x))}
  +\frac{1}{R^{2\alpha}}\|f\|^2_{L^{\infty}(Q_{R}(x))}
  +[f]^2_{C^{0,\alpha}_{\mathcal {X}} (Q_{R}(x))}\Big)\\
  &\leq C\rho^{Q+2+2\alpha}\Big(\frac{1}{R^{4+2\alpha}}\|u\|^2_{L^{\infty}(Q_{R})}
  +\frac{1}{R^{2\alpha}}\|f\|^2_{L^{\infty}(Q_{R})}
  +[f]^2_{C^{0,\alpha}_{\mathcal {X}} (Q_{R})}\Big),
\end{align*}
where ${(X^2u)}_{Q_{R}(t,x)}$ is the mean value of $X^2u$ on $Q_{R}(t,x) $.
Furthermore, it follows from Lemma \ref{A3} and the Jensen inequality that
\begin{align*}
   [X^2u]_{C^{0,\alpha}_{\mathcal {X}}(Q_{R})}
     &\leq C\rho^{Q+2+2\alpha}\Big(\frac{1}{R^{4+2\alpha}}\|u\|^2_{L^{\infty}(Q_{R})}
  +\frac{1}{R^{2\alpha}}\|f\|^2_{L^{\infty}(Q_{R})}
  +[f]^2_{C^{0,\alpha}_{\mathcal {X}} (Q_{R})}\Big)^{\frac{1}{2}}\\
  &\leq C\Big(\frac{1}{R^{2+\alpha}}\|u\|_{L^{\infty}(Q_{R})}
  +\frac{1}{R^{\alpha}}\|f\|_{L^{\infty}(Q_{R})}
  +[f]_{C^{0,\alpha}_{\mathcal {X}} (Q_{R})}\Big).
\end{align*}
\end{proof}
 By the same argument,  Proposition \ref{pro2} gives that the inequality \eqref{bound1} is also valid on $Q^0_R$ as follows.

\begin{Prop}\label{311}
Let $u$ be a  solution of the equation \eqref{linear2}. Then for all $(t_0,x_0)\in[0,T]\times\mathbb{R}^d)$, all  $0<\rho\leq \frac{R}{2}<R<1$ such that $Q^0_R(x_0)\subset [0,T]\times\Omega$, there exists a positive constant $C$ such that
\begin{equation*}
  [X^2 u]_{C^{0,\alpha}_{\mathcal {X}}(Q^0_{\rho})}
   \leq C\Big(\frac{1}{R^{2+\alpha}}\|u\|_{L^{\infty}({Q^0_R})}
  +\frac{1}{R^{\alpha}}\|f\|_{L^{\infty}({Q^0_R})}
  +[f]_{C^{0,\alpha}_{\mathcal {X}} ({Q^0_R})}\Big).
\end{equation*}
\end{Prop}

In the following, we show the periodicity of solutions to the Cauchy problem.
\begin{Lem}\label{periodic}
Let $\mathcal {X}=\{X_1,X_2,\ldots,X_q\}$ be smooth $\mathbb{Z}^d$-periodic vector fields defined on $\mathbb{R}^d$ satisfying H\"{o}rmander's condition.
For the following Cauchy problem
\begin{equation}\label{req}
\left\{
  \begin{aligned}
     & \mathcal {H}u(t,x)=f(t,x),&(t,x)\in (0,T]\times\mathbb{R}^d,  \\
    &u(0,x)=g(x)& x\in\mathbb{R}^d,
  \end{aligned}
\right.
\end{equation}
where $\mathcal {H}$ defined in \eqref{001}, with smooth bounded $\mathbb{Z}^d$-periodic coefficients. Let $b,\ c,\ f,$ and $g$  be smooth bounded and $\mathbb{Z}^d$-periodic in the spatial variable. Suppose the Cauchy problem \eqref{req} is well posed in a class containing the Feynman-Kac solution.
 Then the solution $u(t,x)$ on $\mathbb{R}^d$  is $\mathbb{Z}^d$-periodic with respect to  the  CC-distance, that is
\begin{equation}\label{uxy}
  u(t,x+y)=u(t,x),\ \text{for\ any\ } y\in\mathbb{Z}^d.
\end{equation} Furthermore, it descends uniquely to a solution on $[0,T]\times\mathbb{T}^d$.
\end{Lem}
\begin{proof}
By the Feynman-Kac formula, the probabilistic solution $u(t,x)$ to the equation \eqref{req} is known as
\begin{align}\label{eqFK}
\begin{split}
        u(t,x)
        =\mathbf{E}\Big[
        e^{-\int_0^t c(t-r,Z_r^{x})d r}
          g(Z_t^{x})+\int_0^t
          e^{-\int_0^\tau c(t-r,Z_r^{x})d r}
          f(t-\tau,Z_\tau^{x})d \tau       \Big],
          \end{split}
\end{align}
where $Z_\tau$ is the solution to the following It\^{o} equation
\begin{equation}\label{kac}
dZ^{x}_\tau=-\sum\limits_{k=1}^{q}\sigma^{ik}(Z^{x}_\tau)b_k(Z^{x}_\tau)
-\sum_{k=1}^{q}\sigma^{ik}(Z^{x}_\tau)c_kd\tau
+\sigma(Z^{x}_\tau)dB_\tau,
\end{equation}
starting from $x$ at initial time $t=0$, $c_k=-\sum^{d}_{i=1}\partial_{x_i}\sigma^{ik}(x)$, more details  refer to \cite{SV}.

Note that the diffusion coefficient matrix $\sigma$ is derived by the SDE \eqref{SDE} defined on $[0,T]\times\mathbb{T}^d$ via periodic extension. It is clear that the coefficients $\sigma$ and $b_k$ are Lipschitz continuous on $\mathbb{R}^d$. Furthermore, the Lipschitz continuity of $c_k$ is provided by the smoothness of $\sigma$. Hence we obtain the pathwise uniqueness of strong solutions to the It\^{o} equation \eqref{kac}, it implies that
\begin{equation}\label{y}
        Z_\tau^{x+y}=Z_\tau^{x}+y,
        \quad \text{for\ any\ }0\le \tau\le t,\ y\in\mathbb{Z}^d.
\end{equation}

By using \eqref{eqFK} at the point $x+y$, then we have
$$        u(t,x+y)
        =\mathbf{E}\Big[
        e^{-\int_0^t c(t-r,Z_r^{x+y})d r}
          g(Z_t^{x+y})+\int_0^t
          e^{-\int_0^\tau c(t-r,Z_r^{x+y})d r}
          f(t-\tau,Z_\tau^{x+y})d \tau        \Big].$$
By the translation invariance \eqref{y}, we have
\begin{align*}
u(t,x+y)
        &=\mathbf{E}\Big[
        e^{-\int_0^t c(t-r,Z_r^{x}+y)d r}
          g(Z_t^{x}+y)+\int_0^t
          e^{-\int_0^\tau c(t-r,Z_r^{x}+y)d r}
          f(t-\tau,Z_\tau^{x}+y)d\tau       \Big]\\
          &=\mathbf{E}\Big[
        e^{-\int_0^t c(t-r,Z_r^{x})d r}
          g(Z_t^{x})+\int_0^t
          e^{-\int_0^\tau c(t-r,Z_r^{x})d r}
          f(t-\tau,Z_\tau^{x})d \tau        \Big]\\
          &=u(t,x),
\end{align*}
which is provided by the periodicity of $c,f,g$.
Consider that the solution of the equation \eqref{req} can be expressed by the Feynman-Kac solution \eqref{eqFK}, then $u(t,x)$ also solves the equation \eqref{req} on $[0,T]\times\mathbb{T}^d$.
\end{proof}
\begin{Rem}
Note that the periodicity of the Cauchy problem \eqref{req} cannot proven  via the PDE method, since there is no weak maximum principle results on general H\"{o}rmander operators. In other words, the uniqueness of solutions is not established in this setting. Here we use the Feynman-Kac formula to prove it, the Lipschitz condition guarantees the pathwise uniqueness of solutions.
\end{Rem}
\begin{Lem}\label{equal}
Let $\mathcal {X}=\{X_1,X_2,\ldots,X_q\}$ be smooth $\mathbb{Z}^d$-periodic vector fields defined on $\mathbb{R}^d$ satisfying H\"{o}rmander's condition. Then there exists a constant $C>0$ such that
$$C^{-1}\|u\|_{C^{2,\alpha}_{\mathcal {X}}([0,T]\times\mathbb{T}^d)}\leq\|u\|_{C^{2,\alpha}_{\mathcal {X}}([0,T]\times\mathbb{R}^d)}\leq C\|u\|_{C^{2,\alpha}_{\mathcal {X}}([0,T]\times\mathbb{T}^d)}.$$
\end{Lem}
\begin{proof}
First, we prove  that
\begin{equation}\label{uuu}
\|u\|_{C^{2,\alpha}_{\mathcal {X}}([0,T]\times\mathbb{T}^d)}\leq C\|u\|_{C^{2,\alpha}_{\mathcal {X}}([0,T]\times\mathbb{R}^d)}.
\end{equation}
Denote by $d_{c,\mathbb{T}^d}(x,y):=\inf_{k,l\in\mathbb{Z}^d}d_{c}(x+l,y+k)$ the corresponding $CC$-distance defined on $\mathbb{T}^d$. Then we claim that
\begin{equation}\label{dk}
d_{c}(x,y)=d_{c}(x+k,y+k),\  \text{for\ any\ }k\in\mathbb{Z}^d.
\end{equation}
In fact, if $\gamma$ is an admissible horizontal curve joining $x$ to $y$ with control $\lambda$ satisfying \eqref{gamma}. Set $\gamma_k(t)=\gamma(t)+k$, due to the periodicity of the vector fields, then we have
$$\gamma'_k(t)=\gamma'(t)=\sum^{q}_{j=1}\lambda_j(t)X_j(\gamma(t))=\sum^{q}_{j=1}\lambda_j(t)X_j(\gamma(t)+k)=\sum^{q}_{j=1}\lambda_j(t)X_j(\gamma_k(t)).$$
Thus the curve $\gamma_k(t)$ joins $x+k$ to $y+k$, which implies that the claim \eqref{dk} holds.
Furthermore, we have
\begin{equation}\label{def01}
d_{c,\mathbb{T}^d}(x,y)=\inf_{k,l\in\mathbb{Z}^d}d_{c}(x+l,y+k)=\inf_{k,l\in\mathbb{Z}^d}d_{c}(x,y+k-l)=\inf_{k\in\mathbb{Z}^d}d_{c}(x,y+k).
\end{equation}
Combining  \eqref{8.05} with \eqref{def01},  
it is clearly that the inequality \eqref{uuu} holds.

On the other hand, recall that $\|u\|_{C^{0,\alpha}_{\mathcal {X}}(\Omega'_T)}=[u]_{C^{0,\alpha}_{\mathcal {X}}(\Omega'_T)}+\|u\|_{L^{\infty}(\Omega'_T)}$ for any compact set $\Omega'_T\subseteq[0,T]\times\mathbb{R}^d$. If $d_{P}(x,y)>\epsilon_0$ for any $x,y\in\mathbb{R}^d$ and some $\epsilon_0>0$ small enough chosen below, by the periodicity of $u(t,x)$ given by Lemma \ref{periodic}, then we have
\begin{align*}
  [u]_{C^{0,\alpha}_{\mathcal {X}}(\Omega'_T)} & =\sup_{(t,x)\neq (s,y)}\frac{|u(t,x)-u(s,y)|}{\big(|t-s|+d^2_{c}(x,y)\big)^\frac{\alpha}{2}} \leq 2\epsilon^\alpha_0\|u\|_{L^\infty([0,T]\times\mathbb{R}^d)}\leq C\|u\|_{L^\infty([0,T]\times\mathbb{T}^d)},
\end{align*}
which implies that
\begin{equation}\label{part1}
\|u\|_{C^{2,\alpha}_{\mathcal {X}}([0,T]\times\mathbb{R}^d)}\leq C\|u\|_{L^\infty([0,T]\times\mathbb{T}^d)}.
\end{equation}

 If $d_{P}(x,y)\leq\epsilon_0$ for any $x,y\in\mathbb{R}^d$, then we have $d_{c}(x,y)\leq\epsilon_0$. Now we claim that
\begin{equation}\label{dc}
d_{c,\mathbb{T}^d}(x,y)=d_c(x,y).
\end{equation}
In fact, we proceed by contradiction. Assume that there exists an integer $k_0\neq0$ such that $$d_{c,\mathbb{T}^d}(x,y)=d_c(x,y+k_0)<d_c(x,y)\leq\epsilon_0.$$
Then the triangle inequality yields
\begin{align*}
  d_c(x,y+k_0) & \geq|d_c(y,y+k_0)-d_c(x,y)| =d_c(y,y+k_0)-d_c(x,y),
\end{align*}
when $d_c(y,y+k_0)\geq\epsilon_0$. Since $d_c(y,y+k_0)\geq C^{-1}|k_0|$ given by \eqref{2.7}, then we have
$$\epsilon_0>d_c(x,y+k_0)\geq C^{-1}|k_0|-\epsilon_0.$$
Taking  $\epsilon_0$ small enough such that $\epsilon_0\leq\frac{ C^{-1}|k_0|}{2}$ yields a contradiction. Hence the claim \eqref{dc} is valid.
Moreover, it follows from \eqref{dc} and Lemma \ref{periodic} that
\begin{align*}
  [u]_{C^{0,\alpha}_{\mathcal {X}}([0,T]\times\mathbb{T}^d)} 
   &=\sup_{(t,x)\neq (s,y)}\frac{|u(t,x)-u(s,y)|}{\big(|t-s|+d^2_{c}(x,y)\big)^\frac{\alpha}{2}}=[u]_{C^{0,\alpha}_{\mathcal {X}}([0,T]\times\mathbb{R}^d)}.
\end{align*}
Combined with \eqref{part1}, we have
$$\|u\|_{C^{2,\alpha}_{\mathcal {X}}([0,T]\times\mathbb{R}^d)}\leq C\|u\|_{C^{2,\alpha}_{\mathcal {X}}([0,T]\times\mathbb{T}^d)}.$$
Hence, the proof is completed.
\end{proof}
Now we prove Proposition \ref{prop13} as follows.
\begin{proof}[{\bf{ Proof of Proposition \ref{prop13}}}]
The proof needs to combine the interior estimates with estimate near the bottom boundary  on $[0,T]\times\mathbb{R}^d$.

{\bf{Step 1.}}
If $\hat{u}$ (still denoted $u$) is a solution to the equation \eqref{linear2}, then for any $(t_0,x_0)\in[0,T]\times\Omega'$, $\Omega'\Subset\Omega\subseteq\mathbb{R}^d$,
we first prove that
\begin{equation}\label{aij}
 \|u\|_{C^{2,\alpha}_{\mathcal {X}} ([0,T]\times\Omega')}\leq C\left(\|u\|_{L^{\infty}([0,T]\times\Omega)}+\|f\|_{C^{0,\alpha}_{\mathcal {X}} ([0,T]\times\Omega)}\right),
\end{equation}
where the set $\Omega'$ at least one full period of $\mathbb{T}^d$, and $C$ is a constant depending on $\Omega,\Omega',\mathcal {X},\alpha$ and $C^{0,\alpha}_{\mathcal {X}}$ norms of the coefficients.

It follows from the interior estimates given by Proposition \ref{7210} that there exists $0<R_0<1$ such that $B_{R_0}\subset\Omega$ and for any $0<R\leq \frac{R_0}{2}$,

 \begin{align}\label{inter}
 \begin{split}
   [X^2u]_{C^{0,\alpha}_{\mathcal {X}}(Q_{R})}
     &\leq [X^2u]_{C^{0,\alpha}_{\mathcal {X}}(Q_{R_0/2})}\\
  &\leq C\Big(\frac{1}{R_0^{2+\alpha}}\|u\|_{L^{\infty}(Q_{R_0})}
  +\frac{1}{R_0^{\alpha}}\|f\|_{L^{\infty}(Q_{R_0})}
  +[f]_{C^{0,\alpha}_{\mathcal {X}} (Q_{R_0})}\Big)\\
   &\leq C\Big(\frac{1}{R_0^{2+\alpha}}\|u\|_{L^{\infty}([0,T]\times\Omega)}
  +\frac{1}{R_0^{\alpha}}\|f\|_{L^{\infty}([0,T]\times\Omega)}
  +[f]_{C^{0,\alpha}_{\mathcal {X}} ([0,T]\times\Omega)}\Big).
  \end{split}
\end{align}
The priori estimate  near the bottom boundary given by Proposition \ref{311} gives that
  \begin{align}\label{boundary}
 \begin{split}
   [X^2u]_{C^{0,\alpha}_{\mathcal {X}}(Q^0_{R})}
   &\leq C\Big(\frac{1}{R_0^{2+\alpha}}\|u\|_{L^{\infty}([0,T]\times\Omega)}
  +\frac{1}{R_0^{\alpha}}\|f\|_{L^{\infty}([0,T]\times\Omega)}
  +[f]_{C^{0,\alpha}_{\mathcal {X}} ([0,T]\times\Omega)}\Big).
  \end{split}
\end{align}

 By the finite converge theorem, there exists a sequence of balls $\big\{B_{R_1}(x_1),$ $B_{R_2}(x_2), \ldots, \\ B_{R_n}(x_n)\big\}$ covering the set $\Omega'$, here $\{x_j\}_{1\leq j\leq n}\in\Omega'$ and any ball $B_{2R_j}(x_j)\subset \Omega\subset\mathbb{R}^d$. Set
$$R_0=2\min \{R_1,R_2,\ldots,R_n\}.$$
Then for any $\Omega'\Subset\Omega$, there exists a sequence of points $\{t_1,t_2,\ldots,t_n\}$ such that
\begin{equation}\label{covering}
[0,T]\times\Omega'\subset \bigcup_{1\leq k\leq n}Q_{R_0}(t_k,x_k),
\end{equation}
where $Q_{R_0}(t_k,x_k)$ either is contained within $[0,T]\times\Omega$ or has its center at the points $\{(0,x_k)\}_{1\leq k\leq n}$. Moreover, for $1\leq j\leq n$, it follows from \eqref{inter} and \eqref{boundary} that
$$ [X^2u]_{C^{0,\alpha}_{\mathcal {X}}(Q_{R_0}(t_k,x_k))}\leq C\Big(\frac{1}{R_0^{2+\alpha}}\|u\|_{L^{\infty}([0,T]\times\Omega)}
  +\frac{1}{R_0^{\alpha}}\|f\|_{L^{\infty}([0,T]\times\Omega)}
  +[f]_{C^{0,\alpha}_{\mathcal {X}} ([0,T]\times\Omega)}\Big).$$

 For any $(t,x),\ (s,y)\in [0,T]\times\Omega'$, if $d_P((t,x),(s,y))\geq \frac{R_0}{2}$, then
\begin{equation}\label{801}
  \frac{|X^2u(t,x)-X^2u(s,y)|}{d^\alpha_P((t,x),(s,y))}
  \leq\frac{2}{R_0^\alpha}\|X^2u\|_{L^\infty( [0,T]\times\Omega')}.
\end{equation}

If $d_P((t,x),(s,y))< \frac{R_0}{2}$, then there exists $(\hat{t},\hat{x})\in\{(t_k,x_k)\}_{k\geq1}$ mentioned in \eqref{covering}, such that $(t,x),(s,y)\in Q_{3R_0/4}(\hat{t},\hat{x})\cap \{[0,T]\times\Omega\}$.
It follows from \eqref{inter}, \eqref{boundary} and Lemma \ref{A3} that
\begin{align}\label{U0}
\begin{split}
 &\quad\frac{|X^2u(t,x)-X^2u(s,y)|}{d^\alpha_P((t,x),(s,y))} \\ &\leq[X^2u]_{C^{0,\alpha}_{\mathcal {X}}\left(Q_{3R_0/4}(\hat{t},\hat{x})\right)}\\
  & \leq C\Big(\frac{1}{R_0^{2+\alpha}}\|u\|_{L^\infty([0,T]\times\Omega)}
  +\frac{1}{R_0^{\alpha}}\|f\|_{L^{\infty}([0,T]\times\Omega)}
  +[f]_{C^{0,\alpha}_{\mathcal {X}} ([0,T]\times\Omega)}\Big).
 \end{split}
\end{align}

Since $R_0<1$, combined \eqref{801}  with  \eqref{U0}, we have
\begin{align*}
 [X^2u\big]_{C^{0,\alpha}_{\mathcal {X}} ([0,T]\times\Omega')}
 &=\sup_{(t,x),(s,y)\in [0,T]\times\Omega'}\frac{|X^2u(t,x)-X^2u(s,y)|}{d^\alpha_P((t,x),(s,y))}\\
  & \leq C\Big(\frac{1}{R_0^\alpha}\|X^2u\|_{L^\infty([0,T]\times\Omega')}+\frac{1}{R_0^{2+\alpha}}\|u\|_{L^\infty([0,T]\times\Omega)}
 \\&\quad +\frac{1}{R_0^{\alpha}}\|f\|_{L^{\infty}([0,T]\times\Omega)}
  +[f]_{C^{0,\alpha}_{\mathcal {X}} ([0,T]\times\Omega)}\Big).
\end{align*}
 By a parabolic version of the interpolation inequality given in Corollary \ref{P}, we have
$$\|X^2u\|_{L^\infty([0,T]\times\Omega')}\leq \varepsilon\|X^2u\|_{C^{0,\alpha}_{\mathcal {X}} ([0,T]\times\Omega')}+C\|u\|_{L^{\infty}([0,T]\times\Omega')},$$
for any fixed $0<R_0<1$. If $\varepsilon$ is small enough, then the estimate \eqref{aij} is valid. 

{\bf{Step 2.}} If $u$ is a solution to the equation \eqref{peq}, then we prove the estimate \eqref{3.35}. Without loss of generality, we assume that $g(x)=0$ for any $x\in\mathbb{T}^d$. Indeed, we set $\tilde{f}=f+\mathcal {H}g$. Since $u(t,x)=u(t,x+k)$ given by Lemma \ref{periodic} and $d_c(x,y)=d_c(x+k,y+k)$ given by the claim \eqref{dk}, then for any $x,y\in\Omega'$, we have
\begin{align*}
[u]_{C^{0,\alpha}_{\mathcal {X}}([0,T]\times\{\Omega'+k\})}&=\sup_{(t,x)\neq (s,y)}\frac{|u(t,x+k)-u(s,y+k)|}{\big(|t-s|+d^2_{c}(x+k,y+k)\big)^\frac{\alpha}{2}}
=[u]_{C^{0,\alpha}_{\mathcal {X}}([0,T]\times\Omega')}.
\end{align*}
Since $\mathbb{R}^d=\bigcup_k\{\Omega'+k|k\in\mathbb{Z}^d\}$, by \eqref{aij}, hence we have
$$ \|u\|_{C^{2,\alpha}_{\mathcal {X}} ([0,T]\times\mathbb{R}^d)}\leq C\big(\|u\|_{L^{\infty}([0,T]\times\mathbb{R}^d)}+\|f\|_{C^{0,\alpha}_{\mathcal {X}} ([0,T]\times\mathbb{R}^d)}\big).$$
Due to the periodicity of solutions given by Lemma \ref{periodic}, then we get
\begin{equation}\label{td}
\|u\|_{C^{2,\alpha}_{\mathcal {X}}([0,T]\times\mathbb{T}^d)}\leq C\big(\|f\|_{C^{0,\alpha}_{\mathcal {X}}([0,T]\times\mathbb{T}^d)}+\|u\|_{L^{\infty}([0,T]\times\mathbb{T}^d)}\big).
\end{equation}

For the lower-order terms for the equation \eqref{peq}, by using the interpolation inequality again, since
$b_i,\ c\in C^{0,\alpha}_{\mathcal {X}}([0,T]\times\mathbb{T}^d)$, $i=1,\ldots,q$, we have
$$\|b_iX_iu\|_{C^{0,\alpha}_{\mathcal {X}} ([0,T]\times\mathbb{T}^d)}+\|cu\|_{C^{0,\alpha}_{\mathcal {X}} ([0,T]\times\mathbb{T}^d)}\leq \varepsilon\|u\|_{C^{2,\alpha}_{\mathcal {X}} ([0,T]\times\mathbb{T}^d)}+C\|u\|_{L^{\infty}([0,T]\times\mathbb{T}^d)}.$$
If $\varepsilon$ is small enough, then the estimate \eqref{td} gives that
\begin{align*}
 &\quad\|u\|_{C^{2,\alpha}_{\mathcal {X}} ([0,T]\times\mathbb{T}^d)} \\
 & \leq C\big(\|u\|_{L^{\infty}([0,T]\times\mathbb{T}^d)}+\|f\|_{C^{0,\alpha}_{\mathcal {X}} ([0,T]\times\mathbb{T}^d)}+\|b_iX_iu\|_{C^{0,\alpha}_{\mathcal {X}} ([0,T]\times\mathbb{T}^d)}+\|cu\|_{C^{0,\alpha}_{\mathcal {X}} ([0,T]\times\mathbb{T}^d)} \big)\\
  & \leq C\big(\|u\|_{L^{\infty}([0,T]\times\mathbb{T}^d)}+\|f\|_{C^{0,\alpha}_{\mathcal {X}} ([0,T]\times\mathbb{T}^d)}\big).
\end{align*}

For the general initial condition $g\neq0$, then we have
\begin{equation*}
\|u\|_{C^{2,\alpha}_{\mathcal {X}}([0,T]\times\mathbb{T}^d)}\leq C\big(\|u\|_{L^{\infty}([0,T]\times\mathbb{T}^d)}+\|f\|_{C^{0,\alpha}_{\mathcal {X}}([0,T]\times\mathbb{T}^d)}+\|g\|_{C^{2,\alpha}_{\mathcal {X}}(\mathbb{T}^d)}\big).
\end{equation*}
The proof is completed.
\end{proof}
\subsection{The proof of Theorem \ref{eu}}
\begin{proof}[{\bf{Proof of Theorem \ref{eu}}}]
The main idea of this proof is based on a blend of the Schauder estimates given in Lemma \ref{prop13},  modification method and the continuity method.


\noindent{\bf{Step\ 1.}} First, we show the existence and uniqueness of the solution for the equation
\begin{equation}\label{(3)}
\left\{
  \begin{aligned}
    &\partial_t \tilde{u}-\Delta_\mathcal {X}\tilde{u}=f,&(t,x)\in (0,T]\times\mathbb{T}^d, \\
    &\tilde{u}(0,x)=g(x),&x\in\mathbb{T}^d.
    \end{aligned}
\right.
\end{equation}
Consider the following approximate problem
\begin{equation}\label{apppro}
\left\{
  \begin{aligned}
    &\partial_t u_n-\Delta_\mathcal {X}u_n=f_n,&(t,x)\in (0,T]\times\mathbb{T}^d, \\
    &u_n(0,x)=g(x),&x\in\mathbb{T}^d,
    \end{aligned}
\right.
\end{equation}
where $f_n\in C^\infty([0,T]\times\mathbb{T}^d)$ and $f_n\to f$. By the H\"{o}rmander theorem, there exists a solution $u_n\in C^\infty_b([0,T]\times\mathbb{T}^d)$ for the equation \eqref{apppro}, which implies that $u_n\in C^{2,\alpha}_{\mathcal {X}}([0,T]\times\mathbb{T}^d)$. Applying Proposition \ref{prop13} to the equation \eqref{apppro}, then we have
\begin{align*}
\|u_n\|_{C^{2,\alpha}_{\mathcal {X}}([0,T]\times\mathbb{T}^d)}&\leq C\big(\|u_n\|_{L^{\infty}([0,T]\times\mathbb{T}^d)}+\|f_n\|_{C^{0,\alpha}_{\mathcal {X}}([0,T]\times\mathbb{T}^d)}+\|g(x)\|_{C^{2,\alpha}_{\mathcal {X}}(\mathbb{T}^d)}\big)\leq C.
\end{align*}
It implies that $\{u_n\}$ is a Cauchy sequence in $C^{2,\alpha}_{\mathcal {X}}([0,T]\times\mathbb{T}^d)$. Hence there exists a unique limit $\tilde{u}$ such that $\{u_n\}$ converges locally uniformly to $\tilde{u}$ as $n\to\infty$. Thus we conclude that $\tilde{u}\in C^{2,\alpha}_{\mathcal {X}}([0,T]\times\mathbb{T}^d)$, which is a solution to the equation \eqref{(3)}. Then applying Proposition \ref{prop13} again to the equation \eqref{(3)}, we have
\begin{equation}\label{xue221.6}
   \|\tilde{u}\|_{C^{2,\alpha}_{\mathcal {X}}([0,T]\times\mathbb{T}^d)} \leq C\big(\|f\|_{C^{0,\alpha}_{\mathcal {X}}([0,T]\times\mathbb{T}^d)}+\|g\|_{C^{2,\alpha}_{\mathcal {X}}(\mathbb{T}^d)}\big)\leq C.
\end{equation}
\noindent{\bf{Step\ 2.}} Next, we obtain the existence of the solution for the equation  \eqref{peq} by the continuity method. For any
$\theta\in[0,1]$, set
$$\mathcal {L}_\theta:=\theta\big(\Delta_\mathcal {X}+\sum^{q}_{i=1}b_iX_i+c\big)
+(1-\theta)\sum^{q}_{i=1}\Delta_\mathcal {X}=\Delta_\mathcal {X}+\theta\sum^{q}_{i=1}\big(b_iX_i+c\big).$$
Consider the following equation
\begin{equation}\label{xue3303}
\partial_tu_\theta(t,x)-\mathcal {L}_\theta u_\theta(t,x)=f(t,x).
\end{equation}
In fact, the operator $\mathcal {L}_\theta$ is a bound linear operator from  Banach space $C^{2,\alpha}_{\mathcal {X}}([0,T]\times\mathbb{T}^d)$ to the Banach space $C^{0,\alpha}_{\mathcal {X}}([0,T]\times\mathbb{T}^d)$.
For the equation \eqref{xue3303}, Proposition \ref{prop13} gives that
\begin{equation}\label{xue22126}
   \|u_\theta\|_{C^{2,\alpha}_{\mathcal {X}}([0,T]\times\mathbb{T}^d)} \leq C\big(\|f\|_{C^{0,\alpha}_{\mathcal {X}}([0,T]\times\mathbb{T}^d)}+\|g\|_{C^{2,\alpha}_{\mathcal {X}}(\mathbb{T}^d)}\big)\leq C.
\end{equation}
Consider that $\mathcal {L}_0=\Delta_{\mathcal {X}}$ is an operator from $C^{2,\alpha}_{\mathcal {X}}([0,T]\times\mathbb{T}^d)$ to $C^{0,\alpha}_{\mathcal {X}}([0,T]\times\mathbb{T}^d)$. Then by using the continuity method, the uniform boundness estimate \eqref{xue22126}
implies that $\mathcal {L}_1=\Delta_\mathcal {X}+\sum^{q}_{i=1}b_iX_i+c$ is also an operator from $C^{2,\alpha}_{\mathcal {X}}([0,T]\times\mathbb{T}^d)$ to $C^{0,\alpha}_{\mathcal {X}}([0,T]\times\mathbb{T}^d)$.
Since the existence of the solution $\tilde{u}$ for the equation \eqref{(3)} with the operator $\mathcal {L}_0$ and the estimate \eqref{xue221.6} are obtained in Step 1, then we can obtain the existence of the solution $u$ for the equation \eqref{peq}. The proof is completed.
\end{proof}

\section{WELL-POSEDNESS AND REGULARITY OF THE MFG SYSTEMS}
In this section, we obtain the well-posedness and the regularity of the MFG system \eqref{MFG} given by Theorem \ref{thm3}. To begin with, we need to establish well-posedness and the corresponding estimates for solutions to the HJE \eqref{HJE} and the FPE \eqref{FPE}, respectively, which plays an important role  in the subsequent arguments.
\subsection{Estimates for the HJE}
Let $u$ be a solution of the quasi-linear equation \eqref{HJE}. Set $w(t,x)=e^{-\frac{u(T-t,x)}{2}}$, by using the Hopf-Cole transform, then $w$ is a solution of the following linear equation
\begin{equation}\label{eee}
\left\{
  \begin{aligned}
    &\partial_t w-\Delta_{\mathcal {X}}w+\frac{1}{2}F w=0,& (t,x)\in (0,T]\times\mathbb{T}^d, \\
    &w(0,x)=e^{-\frac{G(x,\bar{m}_T)}{2}},& x\in \mathbb{T}^d.
    \end{aligned}
\right.
\end{equation}
Without loss of generality we assume that $w(0,x)=0$. Indeed, $f=\Delta_{\mathcal {X}}w(0,x)-\frac{1}{2}Fw$ by setting $\tilde{w}(t,x)=w(t,x)-w(0,x)$.
\begin{Prop}\label{PE}
Let $w\in C_{\mathcal {X}}^2([0,T]\times\mathbb{T}^d)$ be a solution of the equation \eqref{eee}, then
\begin{equation}\label{3.01}
\sup_{(t,x)\in[0,T]\times\mathbb{T}^d}|w|\leq\sup_{x\in\mathbb{T}^d}|w(0,x)|+T\sup_{(t,x)\in[0,T]\times\mathbb{T}^d}|f|.
\end{equation}
\end{Prop}
\begin{proof}
First, we note that comparison principle given in Proposition \ref{CP} is also valid for the equation \eqref{eee}.
Then it is not restrictive to assume that $F(t,\bar{m})\geq0$. If not the case, we set $\tilde{w}(t,x)=e^{-\lambda t}w(t,x)$ and $\tilde{u}$ satisfies
$$ \partial_t \tilde{w}-\Delta_{\mathcal {X}}\tilde{w}+\frac{1}{2}\tilde{F}\tilde{w}=f\geq0,$$
with $\tilde{F}=F+2\lambda\geq0$, provided $\lambda$ is suitably chosen.
Set
$$v(t,x)=\sup_{x\in\mathbb{T}^d}|w(0,x)|+t\sup_{(t,x)\in[0,T]\times\mathbb{T}^d}|f(t,x)|.$$
For any $(t,x)\in (t,x)\in[0,T]\times\mathbb{T}^d$, we have
$$\partial_t v-\Delta_{\mathcal {X}}v+\frac{1}{2}Fv=\sup_{(t,x)\in[0,T]\times\mathbb{T}^d}|f(t,x)|+\frac{1}{2}Fv
\geq\sup_{(t,x)\in[0,T]\times\mathbb{T}^d}|f(t,x)|,$$
which implies that $v$ is a lower-semicontinuous
viscosity supersolution for the equation \eqref{eee}.
Similarly, we show that $-v$ is a  upper-semicontinuous
viscosity subsolution for the equation \eqref{eee}.
Consider that for any $x\in \mathbb{T}^d$,
$$-v(0,x)\leq w(0,x)\leq v(0,x),$$
then by using Proposition \ref{CP} for the equation \eqref{eee}, we have
$$\sup_{(t,x)\in[0,T]\times\mathbb{T}^d}|w|\leq|v(t,x)|=\sup_{x\in\mathbb{T}^d}|w(0,x)|
+T\sup_{(t,x)\in[0,T]\times\mathbb{T}^d}|f(t,x)|.$$
The proof is completed.
\end{proof}
Now we prove the existence and uniqueness of solutions for the HJE \eqref{HJE} as follows.
\begin{proof}[{\bf Proof of Proposition \ref{HJEPF}}]
\noindent{\bf{Step\ 1.}}
To prove \eqref{uhje} for the solution $u$, it is necessary to prove the  solution $w$ to the equation \eqref{eee} is a positive solution.
Let $$w^\delta:=\delta -w e^{t\frac{\|F\|_{L^{\infty}([0,T]\times\mathbb{T}^d)}}{2}},\ \text{and}\  \delta:=e^{-\frac{1}{2}\max_{x\in\mathbb{T}^d}|G|}>0.$$
By calculation, we show that $w^\delta$ satisfies
 $$ \partial_t w^\delta-\Delta_{\mathcal {X}}w^\delta+\frac{1}{2}\left(F-\|F\|_{L^{\infty}([0,T]\times\mathbb{T}^d)}\right)w^\delta= \frac{\delta}{2}\left(F-\|F\|_{L^{\infty}([0,T]\times\mathbb{T}^d)}\right)\leq0.$$
Since $w^\delta(0,x)\leq0$, by using Proposition \ref{CP} for $w^\delta$, we get
 $$\max_{(t,x)\in [0,T]\times\mathbb{T}^d} w^\delta(t,x)\leq0.$$
Then $0<\delta\leq w e^{t\frac{\|F\|_{L^{\infty}([0,T]\times\mathbb{T}^d)}}{2}}$.
 Since $F$ is bounded, then $w>0$ for any $t\in(0,T)$.

 Furthermore, consider that $u=-2\ln w$, and
$$  Xu=-2w^{-1}Xw,\quad X^2u=2w^{-1}X^2w+2w^{-2}Xw,$$
combined with the inequality \eqref{21.2} for $w$, we have
\begin{align*}
  \|u\|_{C^{2,\alpha}_{\mathcal {X}}([0,T]\times\mathbb{T}^d)} &\leq C\|w\|_{C^{2,\alpha}_{\mathcal {X}}([0,T]\times\mathbb{T}^d)}\leq C.
\end{align*}

\noindent{\bf{Step\ 2.}}
The solution is unique provided by Proposition \ref{PE}. Indeed, suppose that $w_1$ and $w_2$ are solutions of the equation \eqref{eee} with $g=0$. Let $\tilde{w}=w_1-w_2$, then $\tilde{w}$ is also a solution of \eqref{eee} with $\tilde{w}(0,x)=0$. Then we have $\tilde{w}\equiv0$, it implies that $u$ is unique, the result follows.
\end{proof}
\subsection{Estimates for the FPE}
\begin{Def}{\bf{(Weak solution)}}
We say that $m\in L^1([0,T]\times\mathbb{T}^d)$ is a weak solution to the FPE \eqref{FPE}, if for any test function $\varphi\in C^{\infty}_{0}([0,T]\times\mathbb{T}^d)$, we have
\begin{align}\label{eq4.1}
\begin{split}
& \iint_{[0,T]\times\mathbb{T}^d}-m\partial_t\varphi-\big(\Delta_{\mathcal{X}}\varphi +X \varphi\cdot Xu\big)mdxdt=0.
\end{split}
\end{align}
\end{Def}
 Denote
$H^{k,k}_\mathcal {X}(\Omega_T)$ the isotropic Sobolev space, that is,
$$H^{k,k}_\mathcal {X}(\Omega_T):=\big\{u\in L^2 (\Omega_T):\partial^{j}_{t}X^I u\in L^2(\Omega_T),\ \text{for\ any\ } I\ \text{satisfies}\ |I|+j\leq k\big\}.$$
Denote the function  $\varphi_t(x):=\partial_t\varphi(t,x)$ for any $x\in\mathbb{T}^d$. Since $\int^t_0\varphi(s,x)ds\in C^{\infty}_{0}([0,T]\times\mathbb{T}^d)$, then there is an equivalent definition for the weak solution $m\in H^{1,1}_\mathcal {X}([0,T]\times\mathbb{T}^d)$ to the integral equation \eqref{eq4.1}, if and only if it satisfies
\begin{equation}\label{def2}
  \iint_{[0,T]\times\mathbb{T}^d}\big(\partial_t m \varphi_t-m\Delta_{\mathcal{X}}\varphi_t -mX \varphi_t\cdot Xu\big)e^{-\theta t}dxdt=0,
\end{equation}
for any $\varphi\in C^{\infty}_{0}([0,T]\times\mathbb{T}^d)$, where $\theta$ is any given constant.
\begin{Lem}\label{APP14}
Let $m$ be the weak solution of the FPE \eqref{FPE}, and $u$ be the solution of the HJE \eqref{HJE}. For any  $s,t\in[0,T]$, there is a positive constant $C$ such that

  \noindent\rm{(i)}  $\ d_1(m_t,m_s) \leq C|t-s|^{\frac{1}{2}},$

  \noindent\rm{(ii)}   $\int_{\mathbb{T}^{d}}|x|^2dm_t(x)\leq C\int_{\mathbb{T}^{d}}|x|^2 dm_0+C.$

\end{Lem}
\begin{proof}
First, we prove \rm{(i)} by the same argument as Lemma 3.4 in  \cite{C12}. By the definition \eqref{distance} of $d_1$, the law $\gamma$ of the pair $(Z_t,Z_s)$ belongs to $\Pi(m_t,m_s)$, so
$$ d_1(m_t,m_s) \leq \int_{\mathbb{T}^{d}\times\mathbb{T}^{d}}|x-y|d\gamma(x,y)=\mathbf{E}\big[|Z_t-Z_s|\big],$$
where the state variable $Z_s$ subjects to the SDE \eqref{SDE}.
For instance $t<s$, it follows from the H\"{o}lder inequality, the Jensen inequality and It\^{o} isometry that
\begin{align}\label{7.31}
\begin{split}
\mathbf{E}\big[|Z_s-Z_t|\big]
&\leq \mathbf{E}\Big[\int_t^s |\sigma(Z_\tau)\alpha_\tau|d\tau\Big]
+\Big(\mathbf{E}\Big[\big(\int^s_t|\sigma( Z_\tau)|dB_{\tau}\big)^2\Big]\Big)^{\frac{1}{2}}\\
 &\leq C\|\sigma\|_{L^\infty(\mathbb{T}^d)}\Big(\int_t^s |Xu|^2d\tau\Big)^{\frac{1}{2}}|s-t|^{\frac{1}{2}}
  +\mathbf{E}\Big[\int^s_t|\sigma(Z_\tau)|^2d\tau\Big]^{\frac{1}{2}}\\
  &\leq  C\|\sigma\|_{L^\infty(\mathbb{T}^d)}\|Xu\|_{L^{2}([0,T]\times\mathbb{T}^d)}|t-s|^{\frac{1}{2}},
  \end{split}
  \end{align}
where $\|\sigma\|_{L^\infty(\mathbb{T}^d)}$ and $\|X u\|_{L^\infty([0,T]\times\mathbb{T}^d)}$ are bounded, because of the assumption {\bf{(H1)}} and the uniform estimate \eqref{21.2} holds.

Next, we prove {\rm{(ii)}}. Similar to \eqref{7.31}, we get
\begin{align*}
 \int_{\mathbb{T}^d}|x|^2 dm_t(x)&=\mathbf{E}\left[|Z_t|^2\right]\\
  &  \leq C\mathbf{E}\Big[|Z_0|^2+\Big(\int_{0}^{t}\sigma(Z_\tau)| X u|d\tau\Big)^2
  +\Big(\int_{0}^{t}\sigma(Z_{\tau})dB_{\tau}\Big)^2\Big]\\
 &\leq C\int_{\mathbb{T}^{d}}|x|^2 dm_0+C\|\sigma\|_{L^{\infty}(\mathbb{T}^d)}^2\Big(t^2\| X u\|^2_{L^{\infty}([0,T]\times\mathbb{T}^d)}+t\Big).
  \end{align*}
Hence the estimate {\rm{(ii)}} is valid, and the result follows.
  \end{proof}
 Then by using a variant of Lax-Milgram theorem given in Lemma \ref{App1.1}, we obtain the existence and the  uniqueness  of the solution to the FPE \eqref{FPE}.

\begin{proof}[{\bf Proof of Proposition \ref{3.2}}]
First, recall that the control $\alpha$ is continuous satisfying
$\alpha(t,x)=-Xu,$ where $u$ is the solution to the HJE \eqref{HJE}. We need to show that, for every $\phi\in L^2([0,T]\times\mathbb{T}^d)$, the auxiliary equation
\begin{equation}\label{adj}
\left\{
 \begin{aligned}{ll}
   & \partial_t m-\Delta_\mathcal {X} m-Xu\cdot Xm=\phi,&(t,x)\in [0,T)\times\mathbb{T}^d,\\
   & m(T,x)=0,& x\in\mathbb{T}^d,
\end{aligned}
\right.
\end{equation}
is well-posed in $H^{1,1}_{\mathcal {X}}([0,T]\times\mathbb{T}^d)$ in the weak sense, that is, for any $\varphi\in C^{\infty}_{0}([0,T]\times\mathbb{T}^d)$
\begin{equation*}
\iint_{[0,T]\times\mathbb{T}^d}\big( m_t \varphi_t-Xm X^*\varphi_t- \varphi_tX m \cdot Xu\big)e^{-\theta t}dxdt=\iint_{[0,T]\times\mathbb{T}^d}\phi\varphi_te^{-\theta t}dxdt.
\end{equation*}
The previous well-posedness is proved by standard Hilbert space arguments. In fact, on the space $H^{1,1}_{\mathcal {X}}([0,T]\times\mathbb{T}^d)$, we consider the bilinear form $\left<m,\varphi\right>: H^{1,1}_\mathcal {X}([0,T]\times\mathbb{T}^d)\times C^{\infty}_0([0,T]\times\mathbb{T}^d)\rightarrow \mathbb{R}$ is given by
\begin{equation*}
  \left<m,\varphi\right>:=\iint_{[0,T]\times\mathbb{T}^d}\big( m_t \varphi_t-Xm X^*\varphi_t-\varphi_t X m \cdot Xu\big)e^{-\theta t}dxdt,
\end{equation*}
where $C^{\infty}_0([0,T]\times\mathbb{T}^d)$ is a dense subspace of \(H_{\mathcal{X}}^{1,1}([0,T]\times\mathbb{T}^d)\).
 Then  we claim that the bilinear form $\left<m,\varphi\right> $ satisfies the bounded and coercive condition in the variant of Lax-Milgram theorem given in Lemma \ref{App1.1}, that is, for any $m\in H^{1,1}_\mathcal {X}([0,T]\times\mathbb{T}^d)$, $\varphi\in C^{\infty}_0([0,T]\times\mathbb{T}^d)$, for some $M\geq0,\ \delta>0$ such that
\begin{align}
  \left<m,m\right> &\geq \delta\|m\|^2_{H^{1,1}_\mathcal {X}([0,T]\times\mathbb{T}^d)} ,\label{eq4.3}\\\label{eq4.4}
  |\left<m,\varphi\right>| &\leq M \|m\|_{H^{1,1}_\mathcal {X}([0,T]\times\mathbb{T}^d)}\|\varphi\|_{H^{1,1}_\mathcal {X}([0,T]\times\mathbb{T}^d)} .
\end{align}

For the estimate \eqref{eq4.4}, by the Cauchy-Schwartz inequality and the subelliptic interpolation inequality, it is easily to check  that
\begin{align*}
  |\left<m,\varphi\right> |
   &\leq  M\|m\|_{H^{1,1}_\mathcal {X}([0,T]\times\mathbb{T}^d)}\|\varphi\|_{H^{1,1}_\mathcal {X}([0,T]\times\mathbb{T}^d)} ,
\end{align*}
where $M$ depends on $\|c_k\|_{L^\infty(\mathbb{T}^d)},\ k=\{1,\ldots,q\},$ given by \eqref{ck} and $\|Xu\|_{L^\infty([0,T]\times\mathbb{T}^d)}$.

Now we check the estimate \eqref{eq4.3}. For any $m\in {H^{1,1}_\mathcal {X}([0,T]\times\mathbb{T}^d)}$ and any $\theta$, we obtain
\begin{align*}
  \left<m,m\right> = &\iint_{[0,T]\times\mathbb{T}^d}\big(m_t^{2}-XmX^*m_{t}-m_{t}Xu\cdot  Xm\big)e^{-\theta t}dxdt\\
  =&\iint_{[0,T]\times\mathbb{T}^d}m_t^{2}e^{-\theta t}dxdt+\iint_{[0,T]\times\mathbb{T}^d}XmX m_te^{-\theta t}dxdt\\
  &+\iint_{[0,T]\times\mathbb{T}^d}-\big(\sum^q_{i=1}c_i+Xu\big)Xmm_{t}e^{-\theta t}dxdt\\
   :=& I_1+I_2+I_3.
\end{align*}

By integration by parts, we obtain respectively
\begin{align*}
I_2 &= \frac{1}{2}\iint_{[0,T]\times\mathbb{T}^d}\frac{\partial}{\partial t}\left(|X m|^2 e^{-\theta t}\right)dxdt
       +\frac{\theta}{2}\iint_{[0,T]\times\mathbb{T}^d}|X m|^2 e^{-\theta t}dxdt  \\
   &= \frac{e^{-\theta T}}{2}\int_{\mathbb{T}^d}|X m|^2\big|_{t=T} dx +\frac{\theta}{2}\iint_{[0,T]\times\mathbb{T}^d}|X m|^2 e^{-\theta t}dxdt  \\
   &\geq \frac{\theta}{2}\iint_{[0,T]\times\mathbb{T}^d}|X m|^2 e^{-\theta t}dxdt.
\end{align*}
Since $\|c_i\|_{L^\infty(\mathbb{T}^d)}+\|Xu\|_{L^\infty([0,T]\times\mathbb{T}^d)}\leq C,$ $i=1,\ldots,q$, combining with $\varepsilon$-Cauchy inequality, we obtain
\begin{align*}
I_3   &\geq -C\|Xu\|_{L^\infty([0,T]\times\mathbb{T}^d)}
  \iint_{[0,T]\times\mathbb{T}^d}|Xm m_t|e^{-\theta t}dxdt  \\
   &\geq -\frac{C}{4\varepsilon}\iint_{[0,T]\times\mathbb{T}^d}|Xm|^2 e^{-\theta t}dxdt
       -\varepsilon \iint_{[0,T]\times\mathbb{T}^d}m^2_t e^{-\theta t}dxdt.
  \end{align*}

Together with above these inequalities, we get
\begin{align*}
 \left<m,m\right>
   &\geq (1-\varepsilon)\iint_{[0,T]\times\mathbb{T}^d}m^2_t e^{-\theta t}dxdt +\left(\frac{\theta}{2}-\frac{C}{4\varepsilon}\right)\iint_{[0,T]\times\mathbb{T}^d}|Xm|^2 e^{-\theta t}dxdt\\
&\geq \delta\|m\|^2_{H^{1,1}_\mathcal {X}([0,T]\times\mathbb{T}^d)},
\end{align*}
 the last inequality is valid by taking $\delta>0$ which is large enough, when $\varepsilon>0$ is sufficient small. Then the claim is proved.  Thus the bilinear form $\left<m,\varphi\right> $ is coercive and continuous. Clearly,
 $$ H^{1,1}_\mathcal {X}([0,T]\times\mathbb{T}^d)\ni m\to\iint_{[0,T]\times\mathbb{T}^d}m\phi dxdt\in\mathbb{R}$$ is a continuous linear functional on $C^{\infty}_0([0,T]\times\mathbb{T}^d)$.
By the Lax-Milgram theorem, there  exists a unique $m\in H^{1,1}_\mathcal {X}([0,T]\times\mathbb{T}^d)$ such that for all $\varphi\in C^{\infty}_0([0,T]\times\mathbb{T}^d)$,
 $$\left<m,\varphi\right> =\iint_{[0,T]\times\mathbb{T}^d}\phi \varphi_te^{-\theta t} dxdt.$$

Then we define the linear operator $ T: L^2([0,T]\times\mathbb{T}^d) \to H^{1,1}_\mathcal {X}([0,T]\times\mathbb{T}^d)$ by $T \phi := m$, where $m$ is the unique solution to \eqref{adj}.
Note that  the embedding from $H^1_{\mathcal{X}}([0,T]\times\mathbb{T}^d)$ into \( L^2([0,T]\times\mathbb{T}^d) \) is compact, then \( L^2([0,T]\times\mathbb{T}^d)\to L^2([0,T]\times\mathbb{T}^d)\) is a linear compact operator  (still denoted $T$). Thus the equation
$$
 \partial_t m-\Delta_\mathcal {X} ^*m-{\rm{div}}_{\mathcal {X}^*}(mXu)=0,\ \text{in}\ (0,T]\times\mathbb{T}^d,$$
is equivalent to
$$
(I - T)^*m = 0,
$$
where \( I \) is the identity operator of \( L^2([0,T]\times\mathbb{T}^d) \).

Since \( T \) is compact, the Fredholm alternative applies. Indeed \( I - T \) is a Fredholm operator of index zero (see Lemma 4.45 in \cite{AA02}), which means that $${\rm dim \ ker}(I -  T) ={\rm dim \ ker}(I -  T)^* .$$ In other words, the number of linearly independent solutions of the equation \( (I -  T)^* = 0 \) is equal to the number of linearly independent solutions of the equation \( I -  T = 0 \). Then we must find the number of linearly independent solutions of \( (I -  T )m= 0 \), that is
$$  -\partial_t m-\Delta_{\mathcal {X}}m-Xu\cdot Xm= 0.$$

By Proposition \ref{HJEPF}, the solution \( u \) belongs to \( C^{2,\alpha}_{\mathcal{X}}([0,T]\times\mathbb{T}^d) \) for some \( \alpha \in (0, 1) \). Moreover the operator \(-\Delta_{\mathcal {X}}-Xu\cdot X \) satisfies the strong maximum principle (see \cite{BD01}) and the temporal dimension is non-degenerate. Thus by the considerations above (Fredholm alternative) the equation $(I - T)^* = 0$, and hence the equation \eqref{adj} admits a unique solution \( m \in H^{1,1}_\mathcal{X}([0,T]\times\mathbb{T}^d) \) up to a multiplicative constant.
\end{proof}
\subsection{The proof of Theorem \ref{thm3}}
\begin{proof}[\bf{Proof of Theorem \ref{thm3}}]
The uniqueness of the MFG systems \eqref{MFG} holds depending on the monotonicity of $F$ and $G$ in \textbf{(H4)}, which can refer to Theorem 4.3 in \cite{R18}, and we omit it here. In the following, we only prove the existence, which is based on the Schauder fixed point theorem.

For any probability measures $\mu\in\mathcal {C}$, we associate $m=\psi(\mu)\in\mathcal {C}$ in the following way. Let $u$ be a unique solution to the terminal problem
\begin{equation}\label{3.3}
\left\{
 \begin{aligned}
 &-\partial_t u-\Delta_{\mathcal {X}} u+\frac{1}{2}| X u|^2=F(x,\mu),&(t,x)\in [0,T)\times\mathbb{T}^d,\\
  &u(T,x)=G(x,\mu_T),&x\in\mathbb{T}^d.
 \end{aligned}
\right.
\end{equation}
Then we define $m:=\psi(\mu)$ as the solution of the FPE
\begin{equation}\label{3.4}
\left\{
 \begin{aligned}
   & \partial_t m-\Delta_\mathcal {X}^* m-{\rm{div}}_{\mathcal {X}^*}(mXu)=0,&(t,x)\in [0,T)\times\mathbb{T}^d,\\
  &  m(0,x)=m_0,&x\in\mathbb{T}^d.
 \end{aligned}
\right.
\end{equation}

 First, let us check that $\psi$ is a well-defined from $\mathcal {C}$ to itself. 
 Under  {\bf{(H1)}}-{\bf(H4)},
 Theorem \ref{eu} shows that the HJE \eqref{3.3} has a unique solution $u$ belonging to $C^{2,\alpha}_{\mathcal {X}}([0,T]\times\mathbb{T}^d)$. Moreover, we have an estimate \eqref{21.2}.
Now, we turn to the FPE \eqref{3.4} with the assumption {\bf(H2)}. Since $u\in C^{2,\alpha}_{\mathcal {X}}([0,T]\times\mathbb{T}^d)$, the maps $(t,x)\to Xu(t,x)$ and $(t,x)\to \Delta_{\mathcal {X}}u(t,x)$ belong to $C^{0,\alpha}_{\mathcal {X}}([0,T]\times\mathbb{T}^d)$. From Lemma \ref{APP14}, for any $s, t\in[0,T]$, we have the following estimates
$$d_1(m(t,\cdot),m(s,\cdot)) \leq C|t-s|^{\frac{1}{2}}.$$
 Thus by the definition \eqref{distance}, then $m$ belongs to $\mathcal {C}$, and the mapping $\psi:\mu\rightarrow m:=\psi(\mu)$ is well-defined from $\mathcal {C}$ into itself.

Second, let us check that $\psi$ is a continuous map. Let $\mu_n\in \mathcal {C}$ converge to some $\mu$. Let $(u_n,m_n)$ and $(u,m)$ be the corresponding solutions to \eqref{3.3} and \eqref{3.4}, for any $\varphi\in C^\infty_0([0,T]\times\mathbb{T}^d)$, there holds
\begin{equation}\label{m}
 \iint_{[0,T]\times\mathbb{T}^d}m_n\big(-\partial_t\varphi-\Delta_\mathcal {X}\varphi +Xu_n X\varphi\big) dxd\tau=0.
\end{equation}
By the continuity assumption {\bf{(H3)}} on $F$ and $G$, we get $(t,x)\rightarrow F(x,\mu_n(t))$, $x\rightarrow G(x,\mu_n(T))$ locally uniformly converge to $(t,x)\rightarrow F(x,\mu(t))$, $x\rightarrow G(x,\mu(T))$.
Then one gets the local uniformly convergence of $u_n$ to $u$ by standard arguments of viscosity solutions.
  Since the uniformly estimate \eqref{21.2}, we know that $\{Xu_n\}_n$ are locally uniformly H\"{o}lder continuous and therefore locally uniformly converges to $Xu$.  Let $n\rightarrow\infty$ in \eqref{m}, by the $L^\infty_{loc}$-weak* convergence of $m_n$, and by the convergence $Xu_n\rightarrow Xu$ a.e., then the limit of any converging subsequence of $m_n$ is a weak solution of \eqref{3.4}. But $m$ is the unique weak solution of the equation \eqref{3.4}, which proves that $\{m_n\}$ converges to $m$.

Because $\mathcal {C}$ is compact, the continuous map $\psi$ is compact. We conclude by Schauder fixed point theorem that the compact map $\mu\rightarrow m=\psi(\mu)$ has a fixed point in $\mathcal {C}$. This fixed point $m$ and its corresponding $u$ is a coupling solution of the MFG system \eqref{MFG}. Then the result follows.
\end{proof}
\section{CONCLUSION}
In this paper, we develop a global regularity theory for degenerate parabolic equations \eqref{peq} defined on $\mathbb{T}^d$ induced by H\"{o}rmander vector fields. Since the H\"{o}rmander operators we consider do not admit any group structure, the representations of fundamental solutions are unavailable. Instead, we derive a priori Schauder estimates within the framework of Campanato spaces. As a consequence, we obtain the existence and uniqueness of solutions to the HJE \eqref{HJE} in the intrinsic H\"{o}lder space.

Combined with the well-posedness of the FPE \eqref{FPE}, we obtain the existence and uniqueness of the coupling solutions to the degenerate MFG systems \eqref{MFG}. Here we establish a rigorous characterization of general degenerate MFG systems by means of general H\"{o}rmander  sum-of-squares operators satisfying the natural condition \eqref{con2}.  The coupling solutions can describe the Nash equilibria of a differential game with infinitely many small players, where  the generic player has some forbidden directions because it follows either a dynamic generated by the H\"{o}rmander vector fields.


\addcontentsline{toc}{section}{References} 
\end{document}